\documentclass[11pt,reqno]{amsart}
\usepackage{amsmath, amssymb, amsthm}
\usepackage{url}
\usepackage[breaklinks]{hyperref}

 \renewcommand{\a}{\alpha}

\newcommand{\z}{\zeta}

\renewcommand{\(}{\left\(}
\renewcommand{\)}{\right\)}
\renewcommand{\[}{\left\[}
\renewcommand{\]}{\right\]}
\renewcommand{\i}{\infty}
\numberwithin{equation}{section}
 \theoremstyle{plain}
\newtheorem{theorem}{Theorem}[section]
\newtheorem{lemma}[theorem]{Lemma}
\newtheorem{remark}[]{Remark}

\newtheorem{corollary}[theorem]{Corollary}

   \makeatletter
\def\proof{\@ifnextchar[{\@oproof}{\@nproof}}
\def\@oproof[#1][#2]{\trivlist\item[\hskip\labelsep\textit{#2 Proof of\
#1.}~]\ignorespaces}
\def\@nproof{\trivlist\item[\hskip\labelsep\textit{Proof.}~]\ignorespaces}

\makeatother

\begin{document}


\title[A divisor generating $q$-series identity]{A divisor generating $q$-series identity and its applications to probability theory and random graphs}




\thanks{$2020$ \textit{Mathematics Subject Classification.} Primary 11P84,  33D15; Secondary  05C80,  60F99.\\
\textit{Keywords and phrases.} $q$-series,  partition identities,  generalized divisor functions,  probability distributions,  random graphs} 

\author{Archit Agarwal}
\address{Archit Agarwal, Department of Mathematics, Indian Institute of Technology Indore, Simrol, Indore 453552, Madhya Pradesh, India.}
\email{phd2001241002@iiti.ac.in, archit.agrw@gmail.com}

\author{Subhash Chand Bhoria}
\address{Subhash Chand Bhoria, Pt. Chiranji Lal Sharma Government PG College, Karnal, Urban Estate, Sector-14, Haryana 132001, India.}
\email{scbhoria89@gmail.com}

\author{Pramod Eyyunni}
\address{Pramod Eyyunni,  Department of Mathematics,  Birla Institute of Technology And Science Pilani,  Pilani Campus,  Pilani-333031, Rajasthan, India.}
\email{pramod.eyyunni@pilani.bits-pilani.ac.in,  pramodeyy@gmail.com}

\author{Bibekananda Maji}
\address{Bibekananda Maji, Department of Mathematics, Indian Institute of Technology Indore, Simrol, Indore 453552, Madhya Pradesh, India.}
\email{bibekanandamaji@iiti.ac.in, bibek10iitb@gmail.com}

\maketitle

\begin{abstract}
In I981,  Uchimura studied a divisor generating $q$-series that has applications in probability theory and in the analysis of data structures,  called \emph{heaps}.  
Mainly,  he proved the following identity.  For $|q|<1$,  
\begin{equation*}
\sum_{n=1}^\infty n q^n (q^{n+1})_\infty =\sum_{n=1}^{\infty} \frac{(-1)^{n-1} q^{\frac{n(n+1)}{2} } }{(1-q^n) ( q)_n  } = \sum_{n=1}^{\infty} \frac{ q^n }{1-q^n}.
\end{equation*}
Over the years,  this identity has been generalized by many mathematicians in different directions.  Uchimura himself in 1987,    Dilcher (1995),   Andrews-Crippa-Simon (1997), and recently Gupta-Kumar (2021) found a generalization of the aforementioned identity.  
Any generalization of the right most expression of the above identity,  we name as \emph{divisor-type} sum,  whereas a generalization of the middle expression we say  \emph{ Ramanujan-type} sum,  and any generalization of the left most expression we refer it as \emph{ Uchimura-type} sum.  Quite surprisingly,  Simon,  Crippa and Collenberg (1993) showed that the same divisor generating function has a connection with random acyclic digraphs.   
 One of the main themes of this paper is to study these different generalizations and present a unified theory.   We also discuss applications of these generalized identities in probability theory for the analysis of heaps and random acyclic digraphs.  
\end{abstract}

\tableofcontents

\section{Introduction}
 Ramanujan recorded five interesting $q$-series identities in the unorganized portion of his second notebook \cite[pp.~354--355]{ramanujanoriginalnotebook2}, \cite[pp. 302--303]{ramanujantifr}.  These identities were systematically proved by Berndt \cite[pp. 262--265]{bcbramforthnote}.  Dixit and the fourth author \cite[Theorem 2.1]{DM} obtained a $q$-series identity from which they were able to derive three of these five identities of Ramaujan.  Recently,  the last three authors \cite{BEM21}  established a unified generalization of all of these five identities.  Here we mention one of these five $q$-series identities of Ramanujan,  namely,  
for $  c \in \mathbb{C}, c q^n \neq 1, \ n \geq 1$,   
\begin{equation}\label{entry4}
\sum_{n=1}^{\infty} \frac{(-1)^{n-1} c^n q^{\frac{n(n+1)}{2} } }{(1-q^n) (c q)_n  } = \sum_{n=1}^{\infty} \frac{c^n q^n }{1-q^n}.
\end{equation}
Throughout this paper, we assume $q$ to be a complex number with $|q|<1$, and the notations used above are as follows,
\begin{align*}
&(a)_0=(a;q)_0=1,\\ &(a)_n=(a;q)_n=(1-a)(1-aq)\cdots(1-aq^{n-1}),~n\geq 1,\\ &(a)_\infty=(a;q)_\infty=\lim_{n\rightarrow \infty}(a;q)_n,~\textrm{for} ~|q|<1.
\end{align*}
Uchimura \cite[Equation (3)]{uchimura81} also rediscovered \eqref{entry4}.  
The special case $c=1$ of \eqref{entry4} goes back to Kluyver \cite{kluyver}, 
\begin{equation}\label{Kluyver}
\sum_{n=1}^{\infty} \frac{(-1)^{n-1} q^{\frac{n(n+1)}{2}}}{(1-q^n) ( q)_n  } = \sum_{n=1}^{\infty} \frac{ q^n }{1-q^n}.
\end{equation}
This identity has been independently rediscovered by Fine  \cite[Eqs. (12.4), (12.42)]{fine} and Uchimura \cite[Theorem 2]{uchimura81}.  However,  Uchimura \cite{uchimura81} gave a new infinite series expression for \eqref{Kluyver} by using a sequence of polynomials $U_1(x)=x, U_n(x)= n x^n + (1-x^n) U_{n-1}(x)$. 
These polynomials have connections to the analysis of the data structures called {\it heaps}. Mainly,  he \cite[Theorem 2]{uchimura81}  proved that 
\begin{equation}\label{Uchimura}
\sum_{n=1}^\infty n q^n (q^{n+1})_\infty =\sum_{n=1}^{\infty} \frac{(-1)^{n-1} q^{\frac{n(n+1)}{2} } }{(1-q^n) ( q)_n  } = \sum_{n=1}^{\infty} \frac{ q^n }{1-q^n}.
\end{equation}
This identity has been generalized by many mathematicians and it has a nice partition theoretic interpretation which will be discussed soon.   One of the main aims of this paper is to study different generalizations of  this identity.  Before we  mention these generalizations,  let us fix some terminologies which will be used frequently.   Any generalized form of the left most expression of the above identity we call as {\it Uchimura-type} sum.  However,  any generalization of the middle expression we name as   {\it Ramanujan-type} sum,  and generalization of the right most expression will be referred to as {\it divisor-type} sum since it's a generating function for the well-known divisor function $d(n)$.  

At this moment,  a natural  question arises that `what will be the {\it Uchimura-type} sum for Ramanujan's  identity \eqref{entry4}?' We answer this question in Corollary \ref{Entry 4 with new term}. 

Bressoud and Subbarao \cite{bresub} in 1984, demonstrated a beautiful partition theoretic interpretation,  arising from \eqref{Uchimura}, that connects partitions with the divisor function $d(n)$,  namely, 
\begin{equation}\label{BS}
\sum_{ \pi \in \mathcal{D}(n)  } (-1)^{ \# (\pi)-1 } s(\pi)=d(n),
\end{equation}
where $\mathcal{D}(n)$ is the collection of partitions of $n$ into distinct parts, $\pi$ is an integer partition of $n$, $\#(\pi)$ and $s(\pi)$ is number of parts and smallest part of the partition $\pi$ respectively,  and $d(n)$ is number of positive divisors of $n$.  Later in 1995, Fokkink, Fokkink and Wang \cite{ffw95} also rediscovered \eqref{BS}.  
 Andrews \cite{andrews08},  in his famous paper on ${\rm spt}$-function,  showed that the identity \eqref{BS} can also be obtained by differentitating $q$-analogue of Gauss's theorem \cite[p.~20,  Corollary 2.4]{andrews98}.  
Using purely combinatorical argument,  Bressoud and Subbarao \cite{bresub} further generalized \eqref{BS}. Mainly,  they proved that,  for $ m\in \mathbb{N}\cup \{0\}$, 
\begin{align}\label{BS General}
\sum_{\pi \in \mathcal{D}(n)} (-1)^{\#(\pi)-1}\sum_{j=1}^{s(\pi)}(\ell(\pi)-s(\pi)+j)^m =\sigma_m(n):= \sum_{d|n} d^m,
\end{align}
where $\ell(\pi)$ denotes the largest part of the partition $\pi$.  Letting $m=0$ in the above equation gives \eqref{BS}.
Recently, the last three authors \cite{BEM21} established an identity that is not only a one variable generalization of \eqref{BS General}, but also extends the domain of $m$ from non-negative integers to the set of all integers. To prove their result,  they used a differential and integral operator on a partition theoretic interpretation of Ramanujan's identity \eqref{entry4}.  
The generalized identity \cite[Theorem 2.4]{BEM21} is as  follows, for $m \in \mathbb{Z}$ and $c \in \mathbb{C}$, 
\begin{align}\label{BEM}
\sum_{\pi \in \mathcal{D}(n)} (-1)^{\#(\pi)-1}\sum_{j=1}^{s(\pi)}(\ell(\pi)-s(\pi)+j)^m c^{(\ell(\pi)-s(\pi)+j)} = \sigma_{m, c}(n) :=\sum_{d\mid n} d^m c^d. 
\end{align}
Motivated from the combinatorial argument of Bressoud and Subbarao, the authors \cite[Theorem 2.1]{ABEM} recently showed that the above identity is in fact valid for  any $m \in \mathbb{C}$. 
Moreover,  with the help of a fractional differential operator and analytic continuation the authors gave a simple proof of the generalized identity \eqref{BEM}.  

\subsection{A generalization by Uchimura}
Apart from obtaining the leftmost expression in \eqref{Uchimura},  Uchimura \cite{uchimura87} in 1987,  further gave a one variable generalization of the identity \eqref{Uchimura} in the following way.  For a non-negative integer $m$,  he defined
\begin{align}\label{M_m and K_(m+1)}
M_m := M_m(q) = \sum_{n=1}^{\infty} n^m q^n (q^{n+1})_{\infty}, \quad \text{and} \quad K_{m+1} := K_{m+1}(q)= \sum_{n=1}^{\infty} \sigma_m(n) q^n.
\end{align} 
Then the following properties hold:
\begin{enumerate}
\item $ \displaystyle \exp \left( \sum_{m=1}^{\infty} K_m t^m / m! \right) =  1+ \sum_{m=1}^{\infty} M_m t^m / m!.$
\item Let $Y_m$ be the Bell polynomial defined by
\begin{equation}\label{define Bell polynomial}
Y_m \left( u_1, u_2, \dots, u_m\right) = \sum_{\Pi(m)} \frac{m!}{k_1 ! \dots k_m !} \left( \frac{u_1}{1!} \right)^{k_1} \dots \left( \frac{u_m}{m!} \right)^{k_m},
\end{equation}
where $\Pi(m)$ denotes a partition of $m$ with
$k_1 + 2 k_2 + \cdots + mk_m = m.  $
Then for any $m \geq 1$,   we have 
\begin{equation}\label{uchimura ideantity polynomial version}
M_m = Y_m (K_1, \dots, K_m).             
\end{equation}
\end{enumerate}
This identity was later rediscovered by Andrews,  Crippa and Simon \cite[Theorem 2.6]{andrewssiam1997}.  
Very recently, the authors \cite[Theorem 2.2]{ABEM} gave a one variable generalization of the above result by defining,  for a non-negative integer $m$ and a complex number $c$ with $|cq|<1$,  
\begin{align}
M_{m,c}& := M_{m,c}(q)= \sum_{n=1}^{\infty}n^mc^nq^n(q^{n+1})_{\infty}, \hspace{5mm} \nonumber\\
K_{m+1,c} & :=  K_{m+1,c} (q) = \sum_{n=1}^{\infty} \sigma_{m,c}(n)q^n.  \label{K function}
\end{align}
Then $M_{m, c}$ and $K_{m,c}$ satisfy the following relations:
\begin{align*}
\frac{(q)_\infty}{(cq)_\infty} \exp\bigg(\sum_{m=1}^{\infty}K_{m,c}\frac{t^m}{m!}\bigg)=  \frac{(q)_\infty}{(c q)_\infty} +  \sum_{m=1}^{\infty}M_{m,c}\frac{t^m}{m!}, \quad \text{and}
\end{align*}
for any $m\geq 1$, we have
\begin{align}\label{2nd gen_Uchimura by ABEM}
M_{m,c}=\frac{(q)_\infty}{(cq)_\infty}Y_m(K_{1,c},\cdots, K_{m,c}),
\end{align}
where $Y_m(u_1,u_2,\cdots,u_m)$ denotes the Bell polynomial defined as in \eqref{define Bell polynomial}.
Letting $c=1$ in \eqref{2nd gen_Uchimura by ABEM} immediately yields \eqref{uchimura ideantity polynomial version}.  

Now one can clearly observe that,  in both of the identities  \eqref{uchimura ideantity polynomial version} and \eqref{2nd gen_Uchimura by ABEM},  the left hand side expressions are {\it Uchimura-type} generalized $q$-series,  whereas the right hand side identities are {\it divisor-type} generalized sums.  Thus,  one can ask `what are the {\it Ramanujan-type} generalized sums for the identities \eqref{uchimura ideantity polynomial version} and \eqref{2nd gen_Uchimura by ABEM}?'
We answer this question in Theorem \ref{new expression to ABEM} and Corollary \ref{Uchimura's identity with new expression}.  Interestingly,  we shall see the presence of Eulerian polynomials in the generalized {\it Ramanujan-type} sum for the identity \eqref{2nd gen_Uchimura by ABEM}.

\subsection{A generalization by Dilcher,  and Andrews-Crippa-Simon}
In 1995, Dilcher \cite[Equations (4.3),(5.7)]{dilcher} came up with an interesting generalization of the identity \eqref{Uchimura}, namely, for $k\in \mathbb{N}$,
\begin{align}\label{dilcher 1}
\sum_{n=k}^\infty {n\choose k} q^n(q^{n+1})_\infty=q^{-{k\choose 2}}\sum_{n=1}^\infty \frac{(-1)^{n-1}q^{{n+k \choose 2}}}{(1-q^n)^k (q)_n}=\sum_{j_1=1}^\infty\frac{q^{j_1}}{1-q^{j_1}} \cdots \sum_{j_k=1}^{j_{k-1}}\frac{q^{j_k}}{1-q^{j_k}}.
\end{align}
Substituting $k=1$ in the above identity, we can easily recover \eqref{Uchimura}.  Utilizing \eqref{dilcher 1},  Dilcher \cite[p.~87,  Section 4]{dilcher} tried to establish the following identity of Andrews,  Crippa and Simon \cite[Theorem 2.1]{andrewssiam1997} which also generalizes \eqref{Kluyver}.
For a fixed $k\in \mathbb{N}$, there exist a polynomial $P_k(x_1, x_2,\cdots,  x_k)$ with rational coefficients such that
\begin{align}\label{dilcher 2}
\sum_{n=1}^\infty \frac{(-1)^{n-1}q^{{n+1 \choose 2}}}{(1-q^n)^k (q)_n}=P_k\big(S_0(q), S_1(q),  \dots , S_{k-1}(q)\big):= P_k(q),   
\end{align}
where
\begin{align}\label{dilcher notation for S_s}
S_m(q)=\sum_{n=1}^\infty \frac{n^m q^n}{1-q^n}.
\end{align}
Note that the above generating function $S_m(q)$ is same as the generating function $K_{m-1}(q)$,  defined in \eqref{M_m and K_(m+1)},  of the generalized divisor function $\sigma_{m}(n)$.  An interesting fact is that the series $S_{2m-1}(q)$ also appears in the Fourier series expansion of  Eisenstien series $G_{2m}(z)$ over $SL_{2}(\mathbb{Z})$.  

Now we point out that there is a small error in Dilcher's proof of \eqref{dilcher 2}. In this paper, we shall try to rectify the error and give a correct proof.   A corrected version of Dilcher's identity \cite[Theorem 3]{dilcher} is stated in the main results section (see Theorem \ref{Dilcher's generalization}).  
As the above identity \eqref{dilcher 2} of Andrews,  Crippa and Simon generalizes \eqref{Kluyver},  one may ask
`what will be the {\it Uchimura-type} generalized sum for the identity \eqref{dilcher 2}?' This was answered by Andrews,  Crippa and Simon.   In the same paper \cite[Lemma 2.2]{andrewssiam1997},  they showed that 
\begin{align}\label{ACS}
(q)_\infty \sum_{n=0}^\infty \frac{q^n}{(q)_n}  {k+n-1 \choose k}= \sum_{n=1}^\infty \frac{(-1)^{n-1}q^{{n+1 \choose 2}}}{(1-q^n)^k (q)_n}.
\end{align}
Substituting $k=1$ yields the first equality of \eqref{Uchimura}.  Thus,  the above left hand side sum is the  {\it Uchimura-type} generalized sum for \eqref{dilcher 2}. 

Dixit and the fourth author \cite[Theorem 2.1]{DM} derived a two variable generalization of  Ramanujan's identity \eqref{entry4}.   For $|a|<1,~|b|<1$ and $|c|\leq 1$,  they proved that
\begin{align}\label{dixitmaji}
\sum_{n=1}^\infty \frac{(b/a)_n a^n}{(1-cq^n)(b)_n}=\sum_{m=0}^\infty\frac{(b/c)_mc^m}{(b)_m}\left(\frac{aq^m}{1-aq^m}-\frac{bq^m}{1-bq^m}\right).
\end{align}

\subsection{A generalization by Gupta-Kumar}

Motivated from the identities \eqref{ACS} and \eqref{dixitmaji},   recently Gupta and Kumar \cite[Theorem 1.1]{GK21} established a one variable generalization  of \eqref{ACS}. 
Mainly,  they showed that,  for $|a|<1$ and $k\in \mathbb{N}$,
\begin{align}\label{guptakumar}
\frac{(q)_\infty}{(a)_\infty} \sum_{n=1}^\infty \frac{(a/q)_nq^n}{(q)_n} {k+n-1 \choose k}=-\sum_{n=1}^\infty\frac{(q/a)_n a^n}{(1-q^n)^k (q)_n}. 
\end{align}
Letting $a$ tends to $0$ in \eqref{guptakumar},  one can immediately retrieve \eqref{ACS}.   Gupta and Kumar \cite[Theorem 1.7]{GK21} further obtained a {\it divisor-type} sum for \eqref{guptakumar},  which involves the same generalized divisor function $\sigma_{m,a}(n)$ present in the identity \eqref{BEM}.  They defined
\begin{align*}
\mathfrak{S}_{m,a}(q):=S_{m}(q)-R_{m,a}(q),
\end{align*}
 where $S_{m}(q)$ is same as in \eqref{dilcher notation for S_s} and
\begin{align*}
R_{m,a}(q):= {\rm Li}_{-m}(a)+\sum_{n=1}^\infty \frac{n^m a^nq^n}{1-q^n} = {\rm Li}_{-m}(a)+ K_{m+1,  a}(q),  
\end{align*}
where $K_{m+1,  a}(q)$ is defined as in \eqref{K function} and ${\rm Li}_{-m}(a)$ denotes the well-known polylogarithm function.  For $|a|<1$ and $k\in \mathbb{N}$,  there exists a polynomial $P_k(x_1, x_2,  \dots , x_k)$ with rational coefficients such that 
\begin{align}\label{guptakumar polynomial}
-\sum_{n=1}^\infty \frac{(q/a)_n a^n}{(1-q^n)^k (q)_n}=P_k\big(\mathfrak{S}_{0,a}(q), \mathfrak{S}_{1,a}(q),  \dots , \mathfrak{S}_{k-1,a}(q)\big):=P_k(a, q).
\end{align}
Letting $a \rightarrow 0$ in \eqref{guptakumar polynomial} leads to \eqref{dilcher 2}. 

In the present paper,  our main focus is to find a one variable generalization of Gupta and Kumar's identities  \eqref{guptakumar},  \eqref{guptakumar polynomial}.  As a result,  we obtain a two variable generalization of the identity \eqref{dilcher 2}  of Andrews,  Crippa and Simon.   
Moreover,  we also establish {\it Ramanujan-type} sums for \eqref{uchimura ideantity polynomial version} and \eqref{2nd gen_Uchimura by ABEM} and a {\it Uchimura-type} sum for Ramanujan's identity \eqref{entry4}.  We also rectify an identity of Dilcher \cite[Theorem 3]{dilcher}  and also obtain an elegant one variable generalization of it,  see Theorem \ref{Dilcher's generalization}.

Our paper is organized as follows.  
The main results of our paper are listed in the next section.  However,  Section \ref{application} is devoted to the application of divisor generating $q$-series given by Uchimura and its connection to the theory of random acyclic digraphs developed by Simon-Crippa-Collenberg and Andrews-Crippa-Simon.  A few basic results are stated in Section \ref{Pril}.  
The proofs of our main theorems are given in Section \ref{Proofs}.  In the last section, we discuss a few problems that might be of interest to the readers of the present article.

Now we are ready to state the main results of this article. 
\section{Main Results}\label{section2}

\begin{theorem}\label{General F function}
Let $a,~b,~c,~d$ be complex numbers with $|a|<1,~|c|<1$ and $f(z)$ be an analytic function with power series $\sum_{j=0}^\infty \lambda_j z^j$ under the condition $|z|<1$. Then we have
\begin{align*}
\sum_{n=0}^{\infty}\frac{(b/a)_n a^n}{(d)_n}f(cq^n)=\sum_{j=0}^{\infty}\lambda_j c^j\sum_{n=0}^{\infty}\frac{(ad/b)_n(-b)^nq^{{n\choose 2}+jn}}{(d)_n(aq^j)_{n+1}}. 
\end{align*}
\end{theorem}
By letting $f(z)=(1-z)^{-\alpha}$,  $\alpha \in \mathbb{C}$,  an extended version of the Gupta-Kumar's identity \eqref{guptakumar} is obtained as an immediate corollary of the above theorem.  This corollary is significant enough to be presented as a separate theorem as it will be frequently availed throughout the paper. Prior to revealing this identity,  it is essential to define the binomial coefficient for complex numbers as the quotient of gamma functions.  For a complex number $\alpha$ such that $-\alpha \in \mathbb{C}\backslash \mathbb{N}$, and a natural number $n$, we define the binomial coefficient as follows:
\begin{align}\label{binomial coeff}
{\alpha\choose n}:=\frac{\Gamma(\alpha+1)}{\Gamma(n+1) \Gamma(\alpha-n+1)}.
\end{align}
We now move on to explore several interesting results stemming from Theorem \ref{General F function}.
\begin{theorem}\label{Rajat Rahul Generalization}
Let $a,~b,~c,~d,~\alpha$ be complex numbers with $|a|<1,~|c|<1$. Then
\begin{align*}
\sum_{n=0}^{\infty}\frac{(b/a)_n a^n}{(1-cq^n)^{\alpha}(d)_n}=\sum_{j=0}^{\infty}c^j {\alpha + j -1\choose j}\sum_{n=0}^{\infty}\frac{(ad/b)_n(-b)^nq^{{n\choose 2}+jn}}{(d)_n(aq^j)_{n+1}}. 
\end{align*}
\end{theorem}
By substituting $d=q$ into the above identity, we arrive at a two variable generalization to Gupta and Kumar's identity \eqref{guptakumar}.  Mainly,   we obtain the following identity. 
\begin{corollary}\label{2-var of RR} 
Let $a,~b,~c,~\alpha$ be complex numbers with $|c|<1$, then we have
\begin{align*}
\sum_{n=0}^{\infty}\frac{(b/a)_n a^n}{(1-cq^n)^{\alpha}(q)_n}=\sum_{n=0}^{\infty}c^n{\alpha+n-1\choose n}\frac{(bq^n)_{\infty}}{(aq^n)_{\infty}}. 
\end{align*}
\end{corollary}
Letting $b=q^2$ and $c=q$ in Corollary \ref{2-var of RR} and upon simplification, we get the following immediate extended form of \eqref{guptakumar}. 
\begin{corollary}\label{Gupta Kumar Extension}
For complex numbers $a$ and $\alpha$ with $|a|<1$,  we have
\begin{align*}
\sum_{n=1}^{\infty}\frac{(q/a)_n a^n}{(1-q^n)^{\alpha}(q)_n}=-\frac{(q)_{\infty}}{(a)_{\infty}}\sum_{n=1}^{\infty}q^n{\alpha+n-1\choose \alpha }\frac{(a/q)_{n}}{(q)_{n}}. 
\end{align*}
\end{corollary}
Gupta and Kumar \cite[Theorem 1.1]{GK21}  originally derived the above result under the assumption that $\alpha$ is a natural number. However, it is important to note that their identity is in fact valid for any complex number $\alpha$ as well.

Further,  as an application of Theorem \ref{Rajat Rahul Generalization},  we obtain an elegant two variable generalization of Uchimura's identity \eqref{Uchimura}.
\begin{corollary}\label{Uchimura generalization_2 variable}
For  a complex number $\alpha$ and a natural number $r$, we have
\begin{align}\label{two variable gen_Uchimura}
\sum_{n=1}^\infty \frac{(-1)^{n-1}q^{{n \choose 2}+nr}}{(1-q^n)^\alpha(q)_n} = \sum_{j=r}^\infty  {\alpha + j-r \choose \alpha} q^j (q^{j+1})_\infty.
\end{align}
\end{corollary}
In particular,  when $\alpha =r=1$,  it yields the first equality of Uchimura's  identity \eqref{Uchimura}.  Quite surprisingly,  letting $\alpha = r$,  we recover the first equality of Dilcher's identity \eqref{dilcher 1}.   Thus,  the above identity is a unified generalization of Uchimura's identity \eqref{Uchimura} as well as Dilcher's identity \eqref{dilcher 1}. 
 Now an immediate question arises at this point,  namely,  `what will be the {\it divisor-type} sum for the above generalization?' Unfortunately,  in this paper,  we are unable to answer this question.  

\subsection{A generalization of Dilcher's identity}

Before proceeding to the next theorem, we would like to introduce an identity formulated by Dilcher \cite[Theorem 3]{dilcher}. For $k \in \mathbb{N}$,  the identity is the following:
\begin{align}\label{dilcher's wrong identity}
\sum_{n=1}^\infty \frac{(-1)^{n-1} q^{n+1 \choose 2}}{(1-q^n)^k(q)_n}=\sum_{t=1}^k \left\{\sum_{j=0}^{k-t} {k-1 \choose j+t-1}\frac{s(j+t,t)}{(j+t)!}\right\} U_t(q),
\end{align}
where $s(m,n)$ represents the Stirling numbers of the first kind, and $$U_t(q):=\sum_{n=1}^\infty n^t q^n(q^{n+1})_\infty.$$
However, it is important to note that there is an error in the above identity \eqref{dilcher's wrong identity}.   Instead of $U_t(q)$ on the right side expression of \eqref{dilcher's wrong identity}, it should be the tail $U_{t,j+t}(q)$ of $U_t(q)$,  defined in \eqref{U(m,k)}. 
We not only correct the above identity \eqref{dilcher's wrong identity} but also give a one-variable generalization of it.  Here is the generalized identity.  

 
\begin{theorem} \label{Dilcher's generalization}
 For natural numbers $r$ and $k$, we have
\begin{align}\label{Dilcher's gen_equn}
\sum_{n=1}^\infty \frac{(-1)^{n-1}q^{{n \choose 2}+nr}}{(1-q^n)^k(q)_n} &= \sum_{t=1}^k \sum_{j=0}^{k-t} {k-1 \choose j+t-1} A(j+t,r,t) U_{t,r+j+t-1}(q)\nonumber \\& + \sum_{j=1}^k {k-1 \choose j-1} A(j,r,0) U_{0,r+j-1} (q), 
\end{align}
where
\begin{align}\label{A(j,r,i)}
A(j,r,t):=\sum_{l=t}^j \frac{s(j,l)}{j!}{l \choose t} (1-r)^{l-t}, 
\end{align}
and
\begin{align}\label{U(m,k)}
U_{m,i}(q):=\sum_{n=i}^\infty n^m q^n(q^{n+1})_\infty
\end{align}
is the tail of Uchimura's sum $M_m(q)$ defined in \eqref{uchimura ideantity polynomial version}.
\end{theorem}
Letting $r=1$ in the left-hand side expression of \eqref{Dilcher's gen_equn},  it transforms into the form presented in \eqref{dilcher's wrong identity}.  However,  substituting $r=1$ on the right-hand side of \eqref{Dilcher's gen_equn},  we see that the second finite sum vanishes as $A(j,1, 0)=0$ because  $s(j,0)=0$ for any $j \geq 1$.   Then the first sum becomes 
\begin{align}\label{Dilcher_rhs}
\sum_{t=1}^k \sum_{j=0}^{k-t} {k-1 \choose j+t-1} A(j+t,1,t) U_{t,j+t}(q).
\end{align}
Now from the definition \eqref{A(j,r,i)} of $A(j+t, 1,t)$,  it follows that $A(j+t, 1, t)= \frac{s(j+t,  t)}{(j+t)!}$.  
Thus,  \eqref{Dilcher_rhs} becomes
 \begin{align*}
\sum_{t=1}^k \sum_{j=0}^{k-t} {k-1 \choose j+t-1} \frac{s(j+t,  t)}{(j+t)!} U_{t,j+t}(q),
\end{align*}
which is slightly different from the right hand expression of Dilcher's identity \eqref{dilcher's wrong identity}.  The only difference is that the function $U_t(q)$ in \eqref{dilcher's wrong identity} is being replaced by $U_{t,j+t}(q)$.  


Motivated by the identity \eqref{dilcher's wrong identity},  Dilcher \cite[p.~89]{dilcher} defined an interesting sequence of rational numbers involving Stirling numbers of the first kind.  Mainly,  for $k \in \mathbb{N}$ and $1 \leq t \leq k$,  he defined
\begin{align*}
a(k,  t) = \sum_{j=0}^{k-t} {k-1 \choose j+t-1}\frac{s(j+t,t)}{(j+t)!}. 
\end{align*}
He further proved the following recursive relation for $a(k,t)$ with the help of the recursion formula  for the Stirling numbers of the first kind,  
\begin{align}\label{Dilcher recursion}
a(1,1)= 1,  \quad a(k+1,  t+1) = \frac{k}{k+1}a(k,  t+1) + \frac{1}{k+1} a(k,  t).  
\end{align}
Utilizing the above recursion formula,  Dilcher \cite[Equation (4.10)]{dilcher} obtained an interesting identity,  namely, 
\begin{align}\label{Intersting identity of Dilcher}
\sum_{t=1}^k a(k,  t) = 1.  
\end{align}
Now inspired by the generalization of Dilcher's identity i.e.,  Theorem \ref{Dilcher's generalization},  we define a new sequence of numbers.  
First,  we  rewrite Theorem \ref{Dilcher's generalization} as follows: 
\begin{align*}
\sum_{n=1}^\infty \frac{(-1)^{n-1}q^{{n \choose 2}+nr}}{(1-q^n)^k(q)_n}=\sum_{t=1}^k C(k,r,t)U_{t,r+j+t-1}(q)+ C(k,r,0)U_{0,r+j-1}(q)
\end{align*}
where the coefficients $C(k,r,t)$ and $C(k,r,0)$ defined as
\begin{equation}
\begin{aligned}\label{define C(k,r,t)}
C(k,r,t)& :=\sum_{j=0}^{k-t} {k-1 \choose j+t-1} A(j+t,r,t),\\
C(k,r,0)&:=\sum_{j=1}^{k} {k-1 \choose j-1} A(j,r,0), 
\end{aligned}
\end{equation}
where $A(j,r,t)$ is defined in \eqref{A(j,r,i)}.  
Letting $r=1$,  one can easily verify that $A(j, 1, t) = \frac{s(j,t)}{(j)!}$ and $C(k,1,t)$ is nothing but Dilcher's function $a(k,t)$.  Analysing these coefficients,  we prove the following beautiful recursion formulae: 
\begin{theorem}\label{recurrence relation of C(k,r,t) and A(j,r,t)}
Let  $r,  j,  k \in \mathbb{N}$.  For $1 \leq t \leq j$,   we have
\begin{align}\label{recurrence relation of A(j,r,t)}
A(1,r,1)=1, \hspace{5mm}A(j,r,t)=(j+1)A(j+1,r,t+1)+(r+j-1)A(j,r,t+1),
\end{align}
and
\begin{align}\label{recurrence relation of C(k,r,t)}
C(1,r,1)=1,\hspace{5mm}C(k+1,r,t+1)=\frac{k+1-r}{k+1}C(k,r,t+1)+\frac{1}{k+1}C(k,r,t).  
\end{align}
\end{theorem}
\begin{remark}
When $r=1$,  the recursion \eqref{recurrence relation of A(j,r,t)} reduces to the well-known recursion \cite[p.~214]{comtet} of the Stirling numbers of the first kind,  that is,  
\begin{align*}
s(j,  t) = s (j+1,  t+1) + j s(j,  t+1),
\end{align*}
and the recursion \eqref{recurrence relation of C(k,r,t)} reduces to the recursion \eqref{Dilcher recursion} for $a(k,t)$.  
\end{remark}
Using the definitions \eqref{A(j,r,i)}, \eqref{define C(k,r,t)} of $A(j, r, t)$ and $C(k,r,t)$,  one can also verify that,  for any $r, k \in \mathbb{N}$,  
\begin{align*}
C(k,r,k)=A(k,r,k)=\frac{1}{k!}.
\end{align*}
Further,   we prove the following interesting properties related to $A(j, r, t)$ and $C(k,r,t)$.  
\begin{theorem}\label{summation formula for A[j,r,t) and C(k,r,t)}
For any natural number $j$ and $k$, we have
\begin{align}
& \sum_{t=1}^jA(j,r,t)=\frac{(-1)^{j-1}}{(j-1)!}(r-1)r(r+1)\cdots (r+j-3),    \label{summation formula for A(j,r,t)} \\
& \sum_{t=1}^kC(k,r,t)=\frac{(-1)^{k-1}}{(k-1)!}(r-2)(r-3)(r-4)\cdots (r-k).\label{summation formula for C(j,r,t)}
\end{align}
\end{theorem}
Substituting $r=1$,  we obtain a trivial identity related to Stirling numbers of the first kind and also  recover Dilcher's identity \eqref{Intersting identity of Dilcher}.  
\begin{corollary}
For any $j,  k \in \mathbb{N}$,  we have
\begin{align*}
\sum_{t=1}^j s(j,t) =0,  \quad 
 \sum_{t=1}^k a(k,  t) = 1.  
\end{align*}
\end{corollary}

\subsection{A generalization of identities of Gupta-Kumar and Andrews-Crippa-Simon }

Now we state a one variable generalization of Gupta-Kumar's identity  \eqref{guptakumar polynomial}.  As a result,  we obtain a two variable generalization of the identity \eqref{dilcher 2} of Andrews,  Crippa and Simon.  
\begin{theorem} \label{polynomial version generalization of gupta-kumar}
Let $m \in \mathbb{N}\cup \{0\}$.  Let $a,  c \in \mathbb{C}$ such that $ |ac| <1,  |c q|<1$,  we define $\mathfrak{S}_{m, a, c}(q)$ by
\begin{align} \label{new generalization of divisor function of gupta and kumar}
\mathfrak{S}_{m,a,c}(q):=S_{m,c}(q) - R_{m,ac}(q),  
\end{align}
where
\begin{align}
S_{m,c}(q)& =\sum_{n=1}^\infty \frac{n^m c^n q^n}{1-q^n} = \sum_{n=1}^\infty \sigma_{m,c}(n) q^n, \label{define S_{m,c}} \\ 
 R_{m,ac}(q)& = \textup{Li}_{-m}\left(a c\right)+ \sum_{n=1}^\infty \frac{n^ma^nc^n q^n }{1-q^n}=  \textup{Li}_{-m}\left(a c\right)+  \sum_{n=1}^\infty \sigma_{m,ac}(n) q^n.  \label{define R_{m,ac}}
\end{align}
For $k\in \mathbb{N}$,  there exists a polynomial $P_k(x_1, x_2,\dots, x_k)$ with rational coefficients, such that
\begin{align}\label{two variable_Andrews_Crippa_Simon}
\sum_{n=1}^\infty \frac{(q/a)_{n} a^{n}}{(1-cq^n)^{k+1} (q)_{n-1}}=-\frac{1}{c}\frac{(q)_\infty(ac)_\infty}{(cq)_\infty (a)_\infty} P_{k}(\mathfrak{S}_{0,a,c}(q),\mathfrak{S}_{1,a,c}(q), \dots, \mathfrak{S}_{k-1,a,c}(q)).
\end{align}
\end{theorem}
Substituting $c=1$ in the above theorem and upon simplification,  one can immediately recover the identity \eqref{guptakumar polynomial} of Gupta and Kumar.   Furthermore, letting $a \rightarrow 0$ in \eqref{two variable_Andrews_Crippa_Simon} gives us a new one variable generalization of the identity \eqref{dilcher 2} of Andrews, Crippa and Simon.
\begin{corollary}
For $|c|<1$ and $k\in \mathbb{N}$, there exists a polynomial $P_k(x_1, x_2,\dots, x_k)$ with rational coefficients, such that
\begin{align*}
\sum_{n=1}^\infty \frac{(-1)^{n-1} q^{n+1 \choose 2}}{(1-cq^n)^{k+1} (q)_{n-1}}=\frac{1}{c}\frac{(q)_\infty}{(cq)_\infty } P_{k}(S_{0,c}(q),S_{1,c}(q), \dots, S_{k-1,c}(q)).
\end{align*}
\end{corollary}
In particular when $c=1$,  one can easily verify that the above identity reduces to the identity \eqref{dilcher 2}.  

Now we state a {\it Ramanujan-type} expression for the identity \eqref{2nd gen_Uchimura by ABEM}. Let  $K_{m,c}(q)$ be the generalized divisor generating function defined in \eqref{K function} and $Y_n$ be the Bell polynomial defined in \eqref{define Bell polynomial}.  Let $A_n(x)$ be the Eulerian polynomial of degree $n$ defined as
\begin{align*}
\sum_{n=0}^\infty A_n(x)\frac{t^n}{n!}= \frac{(1-x)e^t}{e^{xt}-xe^t}.
\end{align*}
Then we have the following result.
\begin{theorem} \label{new expression to ABEM} 
For any non-negative integer $m$ and a complex number $c$, we have
\begin{align*}
\sum_{n=1}^\infty n^m c^n q^n(q^{n+1})_\infty = \sum_{n=1}^\infty \frac{(-1)^{n-1} q^{{n \choose 2}}(cq^n)A_m(cq^n)}{(1-cq^n)^{m+1} (q)_{n-1}}=\frac{(q)_\infty}{(cq)_\infty}Y_m(K_{1,c}(q), K_{2,c}(q), \cdots ,K_{m,c}(q)).
\end{align*}
\end{theorem}
Interestingly,  letting $m=1$ in the above theorem gives a {\it Uchimura-type} expression for the Ramanujan's identity \eqref{entry4}.  
\begin{corollary}\label{Entry 4 with new term}
For any complex number $c$, we found that
\begin{align}\label{Uchimura-type for Ramanujan's identity}
\frac{(cq)_\infty}{(q)_\infty}\sum_{n=1}^{\infty} n c^n q^n (q^{n+1})_{\infty} = \sum_{n=1}^{\infty} \frac{(-1)^{n-1} c^n q^{\frac{n(n+1)}{2}}}{(1-q^n) (cq)_n}=\sum_{n=1}^{\infty}\frac{c^n q^n }{1-q^n}. 
\end{align}
\end{corollary}
Substituting $c=1$ in \eqref{Uchimura-type for Ramanujan's identity} reduces to Uchimura's identity \eqref{Uchimura}.  Further,  putting $c=1$ in Theorem \ref{new expression to ABEM} produces a {\it Ramanujan-type} expression for the Uchimura's identity \eqref{uchimura ideantity polynomial version} which seems to be a new expression to the best of our knowledge. 

\begin{corollary} \label{Uchimura's identity with new expression}
For any non-negative integer $m$, we have the following expression
\begin{align*}
\sum_{n=1}^\infty n^m q^n(q^{n+1})_\infty = \sum_{n=1}^\infty \frac{(-1)^{n-1} q^{{n \choose 2}+n}A_m(q^n)}{(1-q^n)^m (q)_n}=Y_m(K_1(q), K_2(q), \dots ,K_m(q)).
\end{align*}
\end{corollary}
Letting $m=1$ in the above corollary and upon simplification,  it yields Uchimura's identity \eqref{Uchimura}.

\section{Applications in probability theory and random graph theory}\label{application}
\subsection{Application of Uchimura's work in probability theory}
 Uchimura \cite{uchimura87} utilized his identity \eqref{uchimura ideantity polynomial version} to give applications in probability theory that can be used in the analysis of data structures called {\it heaps}.  Here we briefly discuss his results.  Uchimura \cite[p.~76]{uchimura87} considered a random variable $X$ such that,  for any $n \in \mathbb{N}\cup \{0\}$ and $0 <q<1$, 
\begin{align}\label{random variable_Uchimura}
\textup{Pr}(X=n)= q^n (q^{n+1})_\infty.
\end{align}
Then the probability generating function is given by 
\begin{align}\label{Probability gen func}
G(x,  q) = \sum_{n=0}^\infty x^n q^n (q^{n+1})_\infty.   
\end{align}
Moreover,  for any  $m \in \mathbb{N}$,  the $m$-th moment is given by
 $$
 E(X^m) = \sum_{n=0}^\infty n^m q^n (q^{n+1})_\infty, 
 $$
which is nothing but the function $M_m$ defined in \eqref{M_m and K_(m+1)}.  Therefore,  utilizing \eqref{uchimura ideantity polynomial version},  we have
$$
E(X^m) = M_m = Y_m (K_1, \dots, K_m).
$$
In particular,  for $m=1$,  it yields
\begin{align}\label{expection_Uchimura}
E(X)=M_1  = \sum_{n=1}^\infty d(n) q^n.
\end{align}
Further,   in the proof of \eqref{uchimura ideantity polynomial version},  Uchimura \cite[p.~74-75]{uchimura87} showed that,  if one writes
\begin{align*}
\log G(e^t,  q) = \sum_{r=1}^\infty h_m \frac{t^m}{m!},
\end{align*}
where the $m$-th coefficient $h_m$ is known as the $m$-th cumulant,  then 
\begin{align}\label{mth cumulant_Uchimura}
h_m = K_m = \sum_{n=1}^\infty \sigma_{m-1} q^n.
\end{align}
We know that the $2$nd cumulant $h_2$ is equals to $\text{Var}(X)$.  Hence,  taking $m=2$,  we have
\begin{align}\label{Variance_Uchimura}
\text{Var}(X) = \sum_{n=1}^\infty \sigma(n) q^n.
\end{align}
Uchimura \cite{uchimura87} further explained a combinatorial interpretation 
of the probability generation function $G(x,  q)$ as defined in \eqref{Probability gen func} and gave immediate applications in the analysis of the data structures  known as heaps.  Mainly,  he studied the average number of exchanges required to  insert an element into a heap.  Interested readers can see 
\cite[Section 3]{uchimura87} for more details. 

\subsection{Works of Simon-Crippa-Collenberg and Andrews-Crippa-Simon}
The divisor generating function appearing in  \eqref{expection_Uchimura} surprisingly also occured in the work of Simon,  Crippa and Collenberg \cite{SCC1993}  in connection with random acyclic digraphs.  To know more about random acyclic digraphs,  see \cite{Palmer85}.  Let the $G_{n,p}$-model of a random acyclic digraph to be the vertex set $V=\{1,  2,  \cdots,  n\}$ and the probability of appearing a directed edge $(i,j)$,  with $1\leq i <j \leq n$,  be $p \in (0,1)$.  
Now consider a random variable $\gamma_{n}^{*}(1)$ that counts the number of vertices that can be connected by a directed path starting from the vertex $1$ (including $1$ itself).  Simon et al.  \cite{SCC1993}   proved that,  for $1 \leq h \leq n$,  
\begin{align}\label{Prob}
\text{Pr}(\gamma_{n}^{*}(1) = h ) = q^{n-h} \prod_{j=1}^{h-1} \left(  1- q^{n-j} \right),
\end{align}
where $q= 1-p$.  In the same paper,  they proved the following interesting result, 
\begin{align}\label{Expectation_Andrews}
\lim_{n \rightarrow \infty} \left( n - E(\gamma_{n}^{*}(1)\right) = \sum_{n=1}^\infty d(n) q^n.
\end{align}
The above identity was further extended by Andrews,  Crippa and Simon  \cite{andrewssiam1997} and they proved that 
\begin{align}\label{Variance_Andrews}
\lim_{n \rightarrow \infty} \text{Var}(\gamma_{n}^{*}(1)) = \sum_{n=1}^\infty \sigma(n) q^n. 
\end{align}
Comparing \eqref{Variance_Uchimura} and \eqref{Variance_Andrews},  one can clearly see that 
$$ \text{Var}(X) = \lim_{n \rightarrow \infty} \text{Var}(\gamma_{n}^{*}(1)),$$ where $X$ is the random variable studied by Uchimura,  defined in \eqref{random variable_Uchimura}.  Here we recall that the generalized divisor generating function \eqref{mth cumulant_Uchimura} is same as the $m$th cumulant $h_m$ with respect to the random variable $X$.  Thus,  a subsequent question raises is that `what will be the interpretation of the generalized divisor generating function \eqref{mth cumulant_Uchimura} with respect to the random variable $\gamma_{n}^{*}(1)$ studied by Simon,  Crippa and Collenberg?'

Around the same time,  Andrews, Crippa, and Simon \cite[Theorem 3.1]{andrewssiam1997} proved a more general result using $q$-series techniques. 
Their result is as follows.  

Let $\{t_n(q)\}$ be a sequence of polynomials satisfying the following recurrence relation,  for $n \geq 1$,  
\begin{align}\label{recurrence_Andrews}
t_n(q)=f(n)+(1-q^{n-1})t_{n-1}(q),~t_0(q)=0,
\end{align}
where $f(n) = \sum_{k\geq 0} c_kn^k \in \mathbb{Q}[n]$ be  a non-zero polynomial.  
Let us recall the polynomial $P_k(q):= P_k\big(S_0(q), S_1(q),\cdots S_{k-1}(q)\big)$ defined in \eqref{dilcher 2}.  Then we have 
\begin{align}\label{andrews crippa simmos's result with limit}
\lim_{n \rightarrow \i} \left( \sum_{i=1}^n f(i)-t_n(q) \right)=\sum_{j\geq 1}h_j P_j(q), 
\end{align}
where coefficients $h_j$ are given by
\begin{align}\label{definition h_j}
h_1=c_0, ~\text{and} ~ \text{for}~j\geq 2, ~ h_j=\sum_{i\geq j-1} (-1)^{i-j+1} {i-1 \choose j-2} i! \sum_{k\geq i} c_k \tilde{s}(k,i),~ 
\end{align}
where $\tilde{s}(k,i)$ denotes the Stirling numbers of the second kind.  

Now letting $t_n(q) :=  E\left(\gamma_{n}^{*}(1)\right)$, they observed that $t_n(q)$ satisfies the recurrence relation \eqref{recurrence_Andrews} with $f(n)=1$ for all $n$.  Thus,  utilizing \eqref{andrews crippa simmos's result with limit},  one can easily obtain \eqref{Expectation_Andrews}.  Further,  as an application of \eqref{andrews crippa simmos's result with limit},  Andrews et al.  also obtained \eqref{Variance_Andrews},  see \cite[Theorem 4.2]{andrewssiam1997}.   

In the same paper,  Andrews,  Crippa and Simon conjectured that,  if $f(n)=(-1)^n$,  then 
\begin{align}\label{Conjecture_ACS}
\lim_{n \rightarrow \infty} \left( \sum_{1 \leq i  \leq n} f(i) - t_n(q)   \right) = \sum_{n=1}^\infty (-1)^n q^{n^2}.  
\end{align}
This conjecture was settled by Bringmann and Jennings-Shaffer \cite{Bringmann}.   Further,  they proved the following result,  for $f(n)= b^n$,  $ b\neq 1$.  
\begin{theorem}
Let $b \in \mathbb{C} \backslash \{1\}$ and $f(n)=b^n$.  Let $t_n(q)$ be a sequence of polynomials defined as in \eqref{recurrence_Andrews}.  Then for $|q| < \min \left(1,  |b|^{-1}  \right)$,   we have
\begin{align*}
\lim_{n \rightarrow \infty} \left( \sum_{1 \leq  i \leq n} b^i - t_n(q)   \right) = \frac{b}{1-b} - \frac{b(q)_\infty}{(b)_\infty}.
\end{align*}
\end{theorem}
Letting $b=-1$ in the above theorem gives \eqref{Conjecture_ACS}.  
Further,  Bringmann and Jennings-Shaffer proved that,  if $f(n)$ is a periodic function with period $N$ and $t_n(q)$ be the sequence of polynomials defined as in \eqref{recurrence_Andrews}.   Let $\zeta_N= e^{\frac{2\pi i}{N}}$.  Then we have
\begin{align}\label{Bringmann and Jennings-Shaffer's result with limit}
\lim_{n \rightarrow \infty} \left( \sum_{1 \leq  i \leq n} f(i) - t_n(q)   \right) =  c_0 S_0(q) - (q)_\infty \sum_{k=1}^{N-1} \frac{c_k}{(\zeta_N^k)_\infty} + \sum_{k=1}^{N-1} \frac{c_k}{1-\zeta_N^k},
\end{align}
where $$c_k = \frac{1}{N} \sum_{j=1}^N f(j) \zeta_N^{(1-j)k}.$$ 
In a recent development, Gupta and Kumar \cite[Theorem 1.11]{GK21} attempted to extend the above result \eqref{Bringmann and Jennings-Shaffer's result with limit}.  However, they made a small error.  In this paper, we rectify it and present a valid generalization of the aforementioned result.

First,  we present a one variable generalization of the identity \eqref{andrews crippa simmos's result with limit} of Andrews,  Crippa and Simon.   Consider the sequence of polynomials $t_n(a,q)$ depending on $a$ and $q$, defined as 
\begin{align}\label{define t_n}
t_n(a,q)=\frac{f(n)-af(n+1)}{1-aq^n} + \left(\frac{1-q^{n-1}}{1-aq^n}\right)t_{n-1}(a,q),~n\geq 1~\text{and}~t_0(a,q)=0,
\end{align}
where $f(n) = \sum_{k\geq 0} c_kn^k \in \mathbb{Q}[n]$ be a polynomial in $n$ with $af(1)=0$.
\begin{theorem}\label{limit in polynomial form}
Let $t_n(a,q)$ be the sequence defined in \eqref{define t_n}.  Then we have
\begin{align*}
\lim_{\ell \rightarrow \infty} \left( \sum_{n=1}^\ell f(n) - \frac{a}{1-a}f(\ell+1) - \left(\frac{1-aq^\ell}{1-a}\right) t_\ell(a,q) \right) = \sum_{j=1}^\infty h_j P_j(a,  q),
\end{align*}
where the  constant $h_j$ is same as in \eqref{definition h_j} 
 and 
$P_j(a, q)$
 is a polynomial in $j$ variables defined in \eqref{guptakumar polynomial}.  
\end{theorem}
\begin{remark} Letting $a = 0$ in Theorem $\ref{limit in polynomial form}$,  one can easily recover the identity \eqref{andrews crippa simmos's result with limit} of Andrews,  Crippa and Simon.  Note that the above result is the corrected version of  Theorem $4.2$ in \cite{GK21}.  
\end{remark}
Now we state a corrected version of another identity of Gupta and Kumar \cite[Theorem 1.11]{GK21}.
\begin{theorem}\label{limit formula for periodic sequence}
Let $f(n)$
be a periodic sequence with period $N$ and $af(1)=0$.  Let $t_n(a,q)$ be the sequence defined in \eqref{define t_n}. 
Then for $|a|<1$ and $|q|<1$, we have
\begin{align*}
& \lim_{\ell \rightarrow \infty} \left( \sum_{n=1}^\ell f(n) - \frac{a}{1-a}f(\ell+1) - \left(\frac{1-aq^\ell}{1-a}\right) t_\ell(a,q) \right) \nonumber  \\ 
& =  c_0 \mathfrak{S}_{0,a}(q) +\sum_{k=1}^{N-1} \frac{c_k}{1-\z_N^k}  -\frac{(q)_\i}{(a)_\i} \sum_{k=1}^{N-1} c_k\frac{(a \zeta_N^k)_\i}{(\z_N^k)\i},
\end{align*}
where
\begin{align*}
c_k=\frac{1}{N}\sum_{j=1}^N f(j) \z_N^{(1-j)k}.  
\end{align*}
\end{theorem}

\begin{corollary}\label{simpler version of limit value}
Let $f(n)$ and $t_n(a,q)$ be same as in Theorem \ref{limit formula for periodic sequence}. Then for $|q|<1$, we have
\begin{align*}
\lim_{\ell \rightarrow \infty} \left( \sum_{n=1}^\ell f(n) - \frac{a}{1-a}f(\ell+1) - \left(\frac{1-aq^\ell}{1-a}\right) t_\ell(a,q) \right) =  \frac{(q)_\i}{(a)_\i} \sum_{n=0}^\i \frac{(a/q)_n q^n}{(q)_n}\sum_{j=1}^N f(j)  \Big\lceil{ {\frac{n+1-j}{N}}} \Big \rceil.
\end{align*}
\end{corollary}
This is the corrected version of Corollary 4.6 in \cite{GK21}.  Letting $a=0$ and $f(n)=(-1)^n$ in Corollary \ref{simpler version of limit value},  one can recover the identity \eqref{Conjecture_ACS} of Bringmann and Jennings-Shaffer.  

In the upcoming section, we gather a few well known results that will be useful throughout this paper.


\section{Preliminary Results}\label{Pril}
For $|z|<1$, the $q$-binomial theorem \cite[p.~8,  Equation (1.32)]{GasperRahman} is given by 
\begin{align}\label{q binomial thm}
\sum_{n=0}^{\infty} \frac{(A)_nz^n}{(q)_n}=\frac{(Az)_{\infty}}{(z)_{\infty}}.
\end{align}
We recall Fine's transformation \cite[p.~359,  (III.1)]{GasperRahman}
\begin{align}\label{Fine's transformation 2phi1}
{}_{2}\phi_{1}\left( \begin{matrix} A,  & B  \\ & C \end{matrix}  ; q; z \right)=\frac{(A z)_{\infty}(B)_{\infty}}{(z)_{\infty}(C)_{\infty}}{}_{2}\phi_{1}\left( \begin{matrix} \frac{C}{B}, & z  \\ & Az \end{matrix}  ; q; B \right),
\end{align}
In particular,  when $z=\frac{C}{AB}$,  it is known as $q$-Gauss sum \cite[p.~354, (II.8)]{GasperRahman}:
\begin{align}\label{q-Gauss sum}
{}_{2}\phi_{1}\left( \begin{matrix} A, & B  \\ & C \end{matrix}  ; q; \frac{C}{AB} \right)= \frac{\left(\frac{C}{A} \right)_\infty  \left(\frac{C}{B} \right)_\infty}{(C)_\infty  \left( \frac{C}{AB}   \right)_\infty}.
\end{align}
The following ${}_{3}\phi_{2}$ transformation formula \cite[p.~359, (III.9)]{GasperRahman} will play a crucial role to prove one of our main theorems:
\begin{align}\label{3_phi_2}
{}_{3}\phi_{2}\left( \begin{matrix} A, & B, & C \\ & D, & E \end{matrix} \, ; q, \frac{DE}{ABC}  \right) = \frac{\big(\frac{E}{A}\big)_\infty\big( \frac{DE}{BC}\big)_{\infty }}{\big(E\big)_\infty \big(\frac{DE}{ABC}\big)_{\infty}} 
{}_{3}\phi_{2}\left(  \begin{matrix}A,& \frac{D}{B} , & \frac{D}{C} \\
& D,& \frac{DE}{BC} \end{matrix} \, ; q, \frac{E}{A} \right).
\end{align}
Now we mention the Chu-Vandermonde identity, 
\begin{align}\label{Chu-Vandermonde}
{k+n-1 \choose k}=\sum_{r=1}^k {n \choose r} {k-1 \choose k-r}.
\end{align}
In the next section,  we present proofs of our main results.

\section{Proofs of main results}\label{Proofs}
\begin{proof}[Theorem \ref{General F function}][]
Using the power series representation of $f(z)$ and interchanging the summations, we can write 
\begin{align*}
\sum_{n=0}^{\infty}\frac{(b/a)_n a^n}{(d)_n}f(cq^n)=\sum_{j=0}^{\infty}\lambda_j c^j\sum_{n=0}^{\infty}\frac{(b/a)_n}{(d)_n}(aq^j)^n.
\end{align*}
The right side of the above expression can be re-written as
\begin{align*}
\sum_{j=0}^{\infty}\lambda_j c^j\sum_{n=0}^{\infty}\frac{(b/a)_n}{(d)_n}(aq^j)^n=\sum_{j=0}^{\infty}\lambda_j c^j \lim_{\gamma \rightarrow 0} {}_{3}\phi_{2}\left( \begin{matrix} q, & b/a, & \gamma d  \\ & d, & \gamma bq^{j+1} \end{matrix}  ; q; aq^j \right). 
\end{align*}
Now we make use of ${}_3\phi_2$ transformation formula \eqref{3_phi_2} to simplify further.  Thus,  we obtain
\begin{align*}
\sum_{j=0}^\infty \lambda_jc^j  \lim_{\gamma \rightarrow 0}  \frac{(\gamma bq^j)_\infty (aq^{j+1})_\infty}{(\gamma bq^{j+1})_\infty (aq^j)_\infty} \sum_{n=0}^\infty \frac{(ad/b)_n (1/\gamma)_n}{(d)_n (aq^{j+1})_n}(\gamma bq^j)^n =  \sum_{j=0}^{\infty}\lambda_j c^j\sum_{n=0}^{\infty}\frac{(ad/b)_n(-b)^nq^{{n\choose 2}+jn}}{(d)_n(aq^j)_{n+1}}. 
\end{align*}
This completes the proof of Theorem \ref{General F function}. 
\end{proof}
\begin{proof}[Theorem \ref{Rajat Rahul Generalization}][]
For any $\alpha \in \mathbb{C}$ and $|z|<1$,  by considering the function $$f(z)=1/(1-z)^\alpha= \sum_{j=0}^\infty {\alpha + j -1\choose j} z^j,  $$ in Theorem \ref{General F function},  where the binomial coefficient is defined as in \eqref{binomial coeff},    we can immediately obtain this result.
\end{proof}

\begin{proof}[Corollary \ref{2-var of RR}][]
Replace $d$ by $q$ in Theorem \ref{Rajat Rahul Generalization} to see that
\begin{align}\label{d=q in cor 2.2}
\sum_{n=0}^{\infty}\frac{(b/a)_n a^n}{(1-cq^n)^\alpha (q)_n}=\sum_{j=0}^{\infty}c^j{\alpha + j -1\choose j} \sum_{n=0}^{\infty}\frac{(aq/b)_n(-b)^nq^{{n\choose 2}+jn}}{(q)_n(aq^j)_{n+1}}.
\end{align}
The right hand side of \eqref{d=q in cor 2.2} can be written as
\begin{align*}
\sum_{j=0}^{\infty}\frac{c^j}{1-aq^j}{\alpha + j -1\choose j} \lim_{\gamma \rightarrow 0} \sum_{n=0}^{\infty}\frac{\left(\frac{aq}{b}\right)_n \left(\frac{b}{\gamma}\right)_n}{(aq^{j+1})_n} \frac{(\gamma q^j)^n}{(q)_n}.
\end{align*}
Now applying $q$-Gauss sum \eqref{q-Gauss sum} in the inner summation and upon simplification,  we get
\begin{align*}
\sum_{j=0}^{\infty}\frac{c^j}{1-aq^j}{\alpha + j -1\choose j}\frac{(bq^j)_\infty}{(aq^{j+1})_\infty}=\sum_{j=0}^{\infty} c^j{\alpha + j -1\choose j}\frac{(bq^j)_\infty}{(aq^{j})_\infty}.
\end{align*}
This finishes the proof.  
\end{proof}

\begin{proof}[Corollary \ref{Gupta Kumar Extension}][]
Using the substitution $b=q^2,~c=q$ in Corollary \ref{2-var of RR}, we get
\begin{align*}
\sum_{n=0}^{\infty}\frac{(q^2/a)_n a^n}{(1-q^{n+1})^{\alpha}(q)_n}=\sum_{n=0}^{\infty}q^n{\alpha+n-1\choose n}\frac{(q^{n+2})_{\infty}}{(aq^n)_{\infty}}.
\end{align*}
Multiply both sides by $a-q$ and replace $\alpha$ by $\alpha + 1$ to derive
\begin{align*}
\sum_{n=0}^{\infty}\frac{(q/a)_{n+1} a^{n+1}}{(1-q^{n+1})^{\alpha}(q)_{n+1}}=-\frac{(q)_{\infty}}{(a)_{\infty}}\sum_{n=0}^{\infty}q^{n+1}{\alpha+n\choose \alpha}\frac{\left(\frac{a}{q}\right)_{n+1}}{(q)_{n+1}}.
\end{align*}
Finally,  replacing the index $n$ of the sum by $n-1$ on both sides yields the desired identity.
\end{proof}

\begin{proof}[Corollary \ref{Uchimura generalization_2 variable}][]
Letting $a \rightarrow 0$ in Theorem \ref{Rajat Rahul Generalization}, we have
\begin{align*}
\sum_{n=0}^\infty \frac{(-1)^n b^n q^{{n \choose 2}}}{(1-cq^n)^{\alpha}(d)_n} &= \sum_{j=0}^\infty c^j {\alpha+j-1 \choose j} \sum_{n=0}^\infty \frac{(-b)^n q^{{n\choose 2}+jn}}{(d)_n}, \\ &= \sum_{j=0}^\infty c^j {\alpha+j-1 \choose j}  \lim_{\gamma \rightarrow 0} {}_2\phi_1 \left( \begin{matrix} b/\gamma, & q  \\ & d \end{matrix}  ~; q~; \gamma q^j \right) \nonumber \\
&=  \sum_{j=0}^\infty c^j {\alpha+j-1 \choose j}  \lim_{\gamma \rightarrow 0} \frac{(b q^j)_\infty (q)_\infty}{(\gamma q^j)_\infty (d)_\infty}  {}_2\phi_1 \left( \begin{matrix} d/q & \gamma q^j  \\ & b q^j \end{matrix}  ~; q~;  q \right) \\
&  = \frac{(q)_\infty}{(d)_\infty} \sum_{j=0}^\infty c^j {\alpha+j-1 \choose j} (bq^j)_\infty \sum_{n=0}^\infty \frac{(d/q)_n q^n}{(bq^j)_n(q)_n}.
\end{align*}
Note that we have used Fine's transformation formula \eqref{Fine's transformation 2phi1} in the penultimate step.   Now replace $\alpha$ by $\alpha+1,~d$ by $c$ and put $b=q^{r+1}$, for some non-negative integer $r$, then we have
\begin{align*}
\sum_{n=0}^\infty \frac{(-1)^n q^{{n+1 \choose 2}+nr}}{(1-cq^n)^{\alpha}(c)_{n+1}} = \frac{(q)_\infty}{(c)_\infty} \sum_{j=0}^\infty c^j {\alpha+j \choose \alpha} (q^{j+r+1})_\infty \sum_{n=0}^\infty \frac{(c/q)_n q^n}{(q^{j+r+1})_n(q)_n}.
\end{align*}
Now changing the index by replacing $n$ by $n-1$ in the left side  and $j$ by $j-r$ in the right side,  we get 
\begin{align*}
\sum_{n=1}^\infty \frac{(-1)^{n-1} q^{{n \choose 2}+(n-1)r} c^r}{(1-cq^{n-1})^{\alpha}(c)_{n}} = \frac{(q)_\infty}{(c)_\infty} \sum_{j=r}^\infty c^j {\alpha+j-r \choose \alpha} (q^{j+1})_\infty \sum_{n=0}^\infty \frac{(c/q)_n q^n}{(q^{j+1})_n(q)_n}.
\end{align*}
Finaly,  replacing $c$ by $q$ in the above identity, we will get the identity \eqref{two variable gen_Uchimura}  as the inner sum of the above summation shall only survive if $n=0$.   
\end{proof}

\begin{proof}[Theorem \ref{Dilcher's generalization}][]
For $k,  r \in \mathbb{N}$,  let us define 
\begin{equation} \label{V_k,r}
V_{k,r}(q) := \sum_{n=1}^\infty \frac{(-1)^{n-1}q^{{n \choose 2}+nr}}{(1-q^n)^k(q)_n}.
\end{equation}
With the help of the binomial theorem, one can verify that
$$
x^{r-1} \sum_{j=1}^k {k-1 \choose j-1} \frac{x^j}{(1-x)^j} = \frac{x^r}{(1-x)^k}.
$$
Substituting $x =q^n$, we have
\begin{align}\label{identity came from binomial}
q^{n(r-1)}\sum_{j=1}^k {k-1 \choose j-1} \frac{q^{nj}}{(1-q^n)^j} = \frac{q^{nr}}{(1-q^n)^k}.
\end{align}
Utilize $\eqref{identity came from binomial}$ in $\eqref{V_k,r}$ to see that
\begin{align*}
V_{k,r}(q)=\sum_{n=1}^\infty \frac{(-1)^{n-1}q^{n\choose 2}}{(q)_n}\sum_{j=1}^k {k-1 \choose j-1} \frac{q^{n(j+r-1)}}{(1-q^n)^j}.
\end{align*}
Now we interchange the summation and use the definition $\eqref{V_k,r}$ to get
\begin{align} \label{Recurrence relation in V}
V_{k,r}(q)=\sum_{j=1}^k {k-1 \choose j-1}V_{j,r+j-1}(q).
\end{align}
Employing Corollary \ref{Uchimura generalization_2 variable},  it can be seen that 
\begin{align}\label{v_j,r+j-1}
V_{j,r+j-1}(q)=\sum_{n=r+j-1}^\infty {n-r+1\choose j} q^n (q^{n+1})_\infty.
\end{align}
Substitute $\eqref{v_j,r+j-1}$ in $\eqref{Recurrence relation in V}$ to see that
\begin{align}\label{V_k,r in product form}
V_{k,r}(q)=\sum_{j=1}^k {k-1 \choose j-1}\sum_{n=r+j-1}^\infty {n-r+1\choose j} q^n (q^{n+1})_\infty.
\end{align}
From the definition of Stirling numbers of the first kind,  one knows that
\begin{align}\label{Binomial and stirling relation}
{n\choose m}=\sum_{l=1}^m \frac{s(m,l)}{m!} n^l.
\end{align} 
We utilize $\eqref{Binomial and stirling relation}$ to write
\begin{align}\label{Binomial and A(j,r,t) relation}
{n-r+1\choose j} =\sum_{l=1}^j \frac{s(j,l)}{j!}(n-r+1)^l & =\sum_{l=1}^j \frac{s(j,l)}{j!}\sum_{t=0}^l{l\choose t}(1-r)^{l-t}n^t 
 =\sum_{t=0}^j A(j,r,t)n^t,
\end{align}
where, $A(j,r,t)$ is defined by
\begin{align*}
A(j,r,t)=\sum_{l=t}^j \frac{s(j,l)}{j!}{l \choose t} (1-r)^{l-t}.
\end{align*}
Here,  we have used the fact that $s(j,0)=0$.  Now employing  $\eqref{Binomial and A(j,r,t) relation}$ in $\eqref{V_k,r in product form}$, we get
\begin{align*}
V_{k,r}(q)=\sum_{j=1}^k {k-1 \choose j-1} \sum_{t=0}^j A(j,r,t) U_{t,r+j-1}(q),
\end{align*}
where $U_{t,r+j-1}(q)$ is the tail of the Uchimura's function defined by
$$
U_{t,r+j-1}(q) = \sum_{n = j+r-1}^\infty n^t q^n (q^{n+1})_\infty.
$$
Now we simplify $V_{k,r}(q)$  further by separating the term corresponding to $t=0$.  Doing this,  we get
\begin{align*}
V_{k,r}(q)& =\sum_{j=1}^k {k-1 \choose j-1} \sum_{t=1}^j A(j,r,t) U_{t,r+j-1}(q) + \sum_{j=1}^k {k-1 \choose j-1}  A(j,r,0) U_{0,r+j-1}(q) \\
& = \sum_{t=1}^k \sum_{j=t}^k {k-1 \choose j-1} A(j,r,t) U_{t,r+j-1}(q) + \sum_{j=1}^k {k-1 \choose j-1}  A(j,r,0) U_{0,r+j-1}(q).
\end{align*}
Finally,  replacing $j$ by $j+t$ in the first sum, one can complete the proof of Theorem \ref{Dilcher's generalization}. 
\end{proof}

\begin{proof}[Theorem \ref{recurrence relation of C(k,r,t) and A(j,r,t)}][]
First,  we shall prove the recurrence \eqref{recurrence relation of A(j,r,t)}. From \cite[Equation (6C), p.~215]{comtet},  we know that
\begin{align}\label{generating function of s(j,l)}
\sum_{l=0}^js(j,l)x^l=(x)_j=x(x-1)(x-2)\cdots(x-j+1).
\end{align}
Successive differentiation of the above equation, $t$ times,  with respect to $x$ yields
\begin{align}\label{A(j,1-x,t) in terms of derivative}
A(j,1-x,t)=\sum_{l=t}^j \frac{s(j,l)}{j!} {l \choose t} x^{l-t} 
& =\frac{1}{j!~t!}\frac{d^t}{dx^t}(x)_j \\
&=\frac{1}{j!~t!}\frac{d^t}{dx^t}\left((x)_{j-1}(x-j+1)\right).  \label{derv}
\end{align}
Using Leibniz rule for differentiation in \eqref{derv}, we get
\begin{align*}
A(j,1-x,t)&=\frac{1}{j!~t!}\left( (x-j+1) \frac{d^t}{dx^t}(x)_{j-1} + t \frac{d^{t-1}}{dx^{t-1}}(x)_{j-1}  \right)\\
&=\frac{(x-j+1)}{j} \frac{1}{(j-1)!~t!} \frac{d^t}{dx^t}(x)_{j-1} + \frac{1}{j~(j-1)!~(t-1)!} \frac{d^{t-1}}{dx^{t-1}}(x)_{j-1}.
\end{align*}
Now utilize \eqref{A(j,1-x,t) in terms of derivative} in the above equation to see that
\begin{align*}
A(j-1,1-x,t-1) = jA(j,1-x,t) + (j-x-1)A(j-1,1-x,t).
\end{align*}
The identity \eqref{recurrence relation of A(j,r,t)} follows by replacing $j,~x,~t$ by $j+1,~1-r,~t+1$, respectively,  in  the above equation.

To prove \eqref{recurrence relation of C(k,r,t)},  we shall start with
\begin{align*}
&\frac{k+1-r}{k+1}C(k,r,t+1)+\frac{1}{k+1}C(k,r,t)\\
&=\frac{k+1-r}{k+1}\sum_{j=0}^{k-t-1} {k-1 \choose j+t} A(j+t+1,r,t+1)+\frac{1}{k+1}\sum_{j=0}^{k-t} {k-1 \choose j+t-1} A(j+t,r,t) \\
&=\frac{k+1-r}{k+1}\sum_{j=0}^{k-t-1} {k-1 \choose j+t} A(j+t+1,r,t+1)\\
&\hspace{1cm}+\frac{1}{k+1}\sum_{j=0}^{k-t} {k-1 \choose j+t-1}\left[(j+t+1)A(j+t+1,r,t+1)+(r+j+t-1)A(j+t,r,t+1) \right].
\end{align*}
Note that to obtain the last step we used the recurrence \eqref{recurrence relation of A(j,r,t)}. 
Using the fact that ${k-1 \choose k}=0$,  the first sum is also valid for $j=k-t$.  Using this and simplifying further,  we obtain
\begin{align*}
& \frac{k+1-r}{k+1}C(k,r,t+1)+\frac{1}{k+1}C(k,r,t) \\ &=\frac{1}{k+1}\sum_{j=0}^{k-t} \left((k+1-r){k-1 \choose j+t} + (j+t+1) {k-1 \choose j+t-1}\right) A(j+t+1,r,t+1)\\
&\hspace{1cm}+\frac{1}{k+1}\sum_{j=1}^{k-t} {k-1 \choose j+t-1} (r+j+t-1)A(j+t,r,t+1),
\end{align*}
where in the last step we used the fact that $A(t,r,t+1)=0$.  By changing the variable $j$ by $j+1$ in the second sum and simplifying further,  we see that 
\begin{align*}
&\frac{k+1-r}{k+1}C(k,r,t+1)+\frac{1}{k+1}C(k,r,t)\\
&=\frac{1}{k+1}\sum_{j=0}^{k-t} \left((k+j+t+1){k-1 \choose j+t} + (j+t+1) {k-1 \choose j+t-1} \right) A(j+t+1,r,t+1)\\
&=\sum_{j=0}^{k-t} {k \choose j+t} A(j+t+1,r,t+1)=C(k+1,r,t+1).
\end{align*}
This completes the proof of \eqref{recurrence relation of C(k,r,t)}.  
\end{proof}

\begin{proof}[Theorem \ref{summation formula for A[j,r,t) and C(k,r,t)}][]
From the definition \eqref{A(j,r,i)} of $A(j, r,t)$,  we can write 
\begin{align*}
\sum_{t=1}^jA(j,r,t)&=\sum_{t=1}^j\sum_{l=1}^j \frac{s(j,l)}{j!}{l \choose t} (1-r)^{l-t}\\
&=\sum_{l=1}^j \frac{s(j,l)}{j!} \sum_{t=1}^l {l \choose t} (1-r)^{l-t}\\
&=\sum_{l=1}^j \frac{s(j,l)}{j!} \left((2-r)^l-(1-r)^l\right)\\
&=\frac{1}{j!}\left(\sum_{l=1}^j s(j,l) (2-r)^l - \sum_{l=1}^j s(j,l)(1-r)^l\right).
\end{align*}
Now utilizing \eqref{generating function of s(j,l)},  one can rewrite the above equation as follows:
\begin{align*}
\sum_{t=1}^jA(j,r,t)&=\frac{1}{j!}\left((2-r)_j - (1-r)_j\right)\\
&=\frac{(-1)^{j-1}}{(j-1)!}(r-1)r(r+1)\cdots (r+j-3).
\end{align*}
This completes the proof of \eqref{summation formula for A(j,r,t)}. 

Now to prove \eqref{summation formula for C(j,r,t)}, we will start with
\begin{align*}
\sum_{t=1}^kC(k,r,t)&=\sum_{t=1}^k\sum_{j=0}^{k-t} {k-1 \choose j+t-1} A(j+t,r,t)\\
&=\sum_{t=1}^k\sum_{j=t}^{k} {k-1 \choose j-1} A(j,r,t)\\
&=\sum_{j=1}^k {k-1 \choose j-1}  \sum_{t=1}^{j} A(j,r,t)\\
&=\sum_{j=1}^k {k-1 \choose j-1} \frac{(-1)^{j-1}}{(j-1)!}(r-1)r(r+1)\cdots (r+j-3)\\
&= \sum_{j=1}^k (-1)^{j-1} {k-1 \choose j-1}  {r+j-3 \choose j-1} \\
& = \sum_{j=0}^{k-1} (-1)^{j} {k-1 \choose j}  {r+j-2 \choose j} \\
& = (-1)^{k-1} {r-2 \choose k-1},
\end{align*}
where to obtain the last equality we used the following binomial identity,  for $a , b \in \mathbb{N}$,   
$$\sum_{j=0}^{a} (-1)^{j} {a \choose j}  {b+j \choose j}  = (-1)^{a} {b \choose a}.$$
This finishes the proof of \eqref{summation formula for C(j,r,t)}.  
\end{proof}

Before proving Theorem \ref{polynomial version generalization of gupta-kumar},  we need the following results.  

\begin{lemma}\label{lemma 5}
For $r\in \mathbb{N},~|a|<1$ and $|c q|<1$, we have
\begin{align*}
\frac{d^{r}}{dx^{r}}\left[\frac{(x ac)_\infty}{(x c q)_\infty}\right]_{x=1} = r!\sum_{n=0}^\infty {n\choose r} \frac{(a/q)_n}{(q)_n}c^n q^n.
\end{align*}
\end{lemma}

\begin{proof}
Using the $q$-binomial theorem \eqref{q binomial thm}, we have 
\begin{align*}
\frac{(x a c)_\infty}{(x c q)_\infty}=\sum_{n=0}^\infty \frac{(a/q)_n}{(q)_n}(x c q)^n.
\end{align*}
Differentiating both sides with respect $x$ successively $r$ times,  we have
\begin{align*}
\frac{d^{r}}{dx^{r}}\left[\frac{(x ac)_\infty}{(x c q)_\infty}\right]&=\sum_{n=0}^\infty  \frac{(a/q)_n}{(q)_n}n(n-1)\cdots(n-r+1)x^{n-r}c^n q^n \\ &= r!\sum_{n=0}^\infty \frac{(a/q)_n}{(q)_n} {n \choose r} x^{n-r}c^n q^n.
\end{align*}
Letting $x \rightarrow 1$ in the above equation, we get the result.
\end{proof}

%
%
Now let us consider the following function
\begin{align} \label{new t function}
T_{r,a,c}=T_{r,a,c}(x,q):=\sum_{n=1}^\infty \frac{c^r q^{nr}}{(1-x cq^{n})^r}-\sum_{n=0}^\infty \frac{a^r c^r q^{nr}}{(1-x acq^{n})^r}.
\end{align}
One can easily verify that
\begin{align}\label{differentiation of new t function}
\frac{d}{dx}T_{r,a,c}(x,q)=rT_{r+1,a,c}(x,q).
\end{align}

\begin{lemma} \label{lemma 6}
For each $i\in \mathbb{N}$, $a,c \in \mathbb{C}$ and $|q|<1$ there exists a polynomial $N_i(x_1, x_2, \dots, x_i)\in \mathbb{Q}[x_1, x_2,  \dots, x_i]$  of degree $i$ such that 
\begin{align*}
\frac{d^i}{dx^i}\left[ \frac{(x ac)_\infty}{(x c q)_\infty}\right]=\frac{(x ac )_\infty}{(x cq)_\infty} N_i(T_{1,a,c},T_{2,a,c},\dots, T_{i,a,c}).
\end{align*}
\end{lemma}

\begin{proof}
We start with the following product representation
$$
\frac{(x ac)_\infty}{(xc q)_\infty} = \prod_{n=0}^\infty  \frac{1-x ac q^{n}}{1-x  cq^{n+1}}.
$$
Then differentiating both sides with respect to $x$,  we see that
\begin{align*}
\frac{d}{dx}\left[ \frac{(x ac)_\infty}{(x c q)_\infty}\right]&=\frac{d}{dx}\left[\exp \left\{ \log \left( \prod_{n=0}^\infty  \frac{1-x ac q^{n}}{1-x  cq^{n+1}}\right) \right\}  \right]\\ &= \prod_{n=0}^\infty  \frac{1-x ac q^{n}}{1-x  cq^{n+1}} \frac{d}{dx} \left[\sum_{n=0}^\infty \log \left( \frac{1-x ac q^{n}}{1-x c q^{n+1}}\right) \right]\\ &= \frac{(x ac)_\infty}{(x c q)_\infty} T_{1,a,c}(x, q)\\ &= \frac{(x ac)_\infty}{(x c q)_\infty} N_1(T_{1,a,c}),
\end{align*}
with $N_1(x_1)=x_1$ and $T_{1,a,c}$ is defined as in \eqref{new t function}.  Further,   differentiating the above equality with respect to $x$,  it reduces to 
\begin{align*}
\frac{d^2}{dx^2}\left[ \frac{(x ac)_\infty}{(x c q)_\infty}\right]=\frac{(x ac)_\infty}{(x c q)_\infty} T_{1,a,c}^2(x, q)+ \frac{(x ac)_\infty}{(x c q)_\infty} \frac{d}{dx} T_{1,a,c}(x, q).
\end{align*}
Now utilizing \eqref{differentiation of new t function},  one sees that 
\begin{align*}
\frac{d^2}{dx^2}\left[ \frac{(x ac )_\infty}{(x c q)_\infty}\right] = \frac{(x ac )_\infty}{(x c q)_\infty} T_{1,a,c}^2(x, q)+ \frac{(x ac)_\infty}{(x c q)_\infty}  T_{2,a,c}(x, q) = \frac{(x ac)_\infty}{(x c q)_\infty} N_2(T_{1,a,c}, T_{2,a,c}),
\end{align*}
with $N_2(x_1, x_2) =x_1^2+x_2$.  Continuing differentiating with respect to $x$ and using the relation \eqref{differentiation of new t function},  one can complete the proof of Corollary \ref{lemma 6}.  
\end{proof}

\begin{lemma}\label{lemma 7} Let $\mathfrak{S}_{m,a,c}(q)$ and $T_{r, a,c}(x,q)$  be the functions defined as in \eqref{new generalization of divisor function of gupta and kumar} and \eqref{new t function},  respectively. For every $r\in \mathbb{N}$,  we have
\begin{align*}
T_{r,a,c}(1,q) = \sum_{h=0}^{r-1}Q_{h,r} \mathfrak{S}_{h,a,c}(q),
\end{align*}
where $Q_{h, r} \in \mathbb{Q}$.  
\end{lemma}

\begin{proof}
From \eqref{new t function}, we have 
\begin{align}\label{new t function at 1}
T_{r,a,c}(1,q)=\sum_{n=1}^\infty \frac{c^r q^{nr}}{(1- cq^{n})^r}-  \sum_{n=0}^\infty \frac{a^r c^r q^{nr}}{(1- acq^{n})^r}.
\end{align}
Using binomial theorem twice,  for $|cq|<1$,  one can show that 
\begin{align}\label{new t function in 2nd form}
\sum_{n=1}^\infty \frac{c^r q^{nr}}{(1- cq^{n})^r} & =  \sum_{n=1}^\infty \frac{ c q^n (1- (1- cq^n))^{r-1}}{(1- cq^{n})^r}  \nonumber \\
& =   \sum_{n=1}^\infty \frac{c q^n}{ (1- cq^n )^r} \sum_{j=0}^{r-1} { {r-1} \choose j} (-1)^j (1 -c q^n)^j \nonumber \\
& = \sum_{j=0}^{r-1} { {r-1} \choose j} (-1)^j \sum_{n=1}^\infty \frac{c q^n}{(1- cq^n)^{r-j}} \nonumber \\
& = \sum_{j=0}^{r-1} { {r-1} \choose j} (-1)^j \sum_{n=1}^\infty \sum_{m=0}^\infty { r-j+m-1 \choose m} ( c q^n)^{m+1} \nonumber \\
& =  \sum_{j=0}^{r-1} { {r-1} \choose j}  \frac{(-1)^j}{(r-j-1)!} \sum_{n=1}^\infty \sum_{m=1}^\infty (c q^n)^m m (m+1) \cdots (m+r-j-2) \nonumber \\
& =  \sum_{j=0}^{r-1} { {r-1} \choose j}  \frac{(-1)^{r-1}}{(r-j-1)!} \sum_{n=1}^\infty \sum_{m=1}^\infty (c q^n)^m (-m) (-m-1) \cdots (-m-r+j+2).
\end{align}
Now using the definition of the Stirling numbers of the first kind,  that is,
\begin{align*}
x(x-1)\cdots(x-j+1)=\sum_{h=0}^j s(j,  h)x^h,
\end{align*}
the above expression in \eqref{new t function in 2nd form} can be written as follows: 
\begin{align}
 & \sum_{j=0}^{r-1} { {r-1} \choose j}  \frac{(-1)^{r-1}}{(r-j-1)!} \sum_{n=1}^\infty \sum_{m=1}^\infty  (c q^n)^m \sum_{h=0}^{r-j-1} s(r-j-1,  h) (-m)^h \nonumber \\
  & = \sum_{j=0}^{r-1} { {r-1} \choose j}  \frac{(-1)^{r-1}}{(r-j-1)!}  \sum_{h=0}^{r-j-1}(-1)^h s(r-j-1,  h) \sum_{n=1}^\infty \sum_{m=1}^\infty  m^h (c q^n)^m \nonumber \\
  & = \sum_{j=0}^{r-1} { {r-1} \choose j}  \frac{(-1)^{r-1}}{(r-j-1)!}  \sum_{h=0}^{r-j-1}(-1)^h s(r-j-1,  h) S_{h,c}(q) \nonumber \\
  & = \sum_{h=0}^{r-1}  \sum_{j=0}^{r-h-1} { {r-1} \choose j}  \frac{(-1)^{h+r-1}}{(r-j-1)!} s(r-j-1, h) S_{h,c}(q),  \nonumber \\
  & = \sum_{h=0}^{r-1} Q_{h,  r} S_{h,c}(q), \label{first}
\end{align}
 where $S_{h,c}(q)$ is defined as in \eqref{define S_{m,c}} and the rational numbers $Q_{h,r}$ are defined by
\begin{align*}
Q_{h,r} = \sum_{j=0}^{r-h-1} \frac{(-1)^{r+h-1}}{(r-j-1)!} {r-1 \choose j} s(r-j-1,  h).  
\end{align*}
In a similar way,   we can prove that the second sum in \eqref{new t function at 1} is 
\begin{align}\label{2nd sum}
 \sum_{n=0}^\infty \frac{a^r c^r q^{nr}}{(1- acq^{n})^r} 
 =  \sum_{h=0}^{r-1} Q_{h,  r}  \sum_{n=0}^\infty \sum_{m=1}^\infty m^h ( a c q^n)^m.
\end{align}
Now one can show that 
\begin{align}\label{second sum}
 \sum_{n=0}^\infty \sum_{m=1}^\infty m^h ( a c q^n)^m = \textup{Li}_{-h}( ac) + S_{h,  ac}(q). 
\end{align}
Thus,  substituting \eqref{second sum} in \eqref{2nd sum},  we have
\begin{align}
\sum_{n=0}^\infty \frac{a^r c^r q^{nr}}{(1- acq^{n})^r}  & = \sum_{h=0}^{r-1} Q_{h, r} R_{h,  ac}(q),  \label{second}
\end{align}
where $R_{h,  ac}(q)$ is  as  defined in \eqref{define R_{m,ac}}.  
Utilizing \eqref{first} and \eqref{second} in \eqref{new t function at 1} and together with the definition \eqref{new generalization of divisor function of gupta and kumar} of $\mathfrak{S}_{h,a,c}(q)= S_{h,c}(q) - R_{h, ac}(q)$,  we finish the proof of this lemma.  

\end{proof}

\begin{proof}[Theorem \ref{polynomial version generalization of gupta-kumar}][]
Substituting $b=q^2$ and taking $\alpha$ to be a natural number, say $k+1$, in Corollary \ref{2-var of RR}, we have
\begin{align*}
\sum_{n=0}^{\infty}\frac{(q^2/a)_n a^n}{(1-cq^n)^{k+1}(q)_n} &=\sum_{n=0}^{\infty}c^n{k+n \choose n}\frac{(q^{n+2})_{\infty}}{(aq^n)_{\infty}}\\
&= \frac{(q)_\infty}{(a/q)_\infty} \sum_{n=1}^{\infty} c^{n-1} {k+n-1 \choose k}\frac{(a/q)_{n}}{(q)_{n}}.
\end{align*}
Further,  replace $c$ by $cq$ and use Chu-Vandemonde identity \eqref{Chu-Vandermonde} in the above equation to get
\begin{align*}
\sum_{n=0}^{\infty}\frac{(q^2/a)_n a^n}{(1-cq^{n+1})^{k+1}(q)_n} &=\frac{1}{cq}\frac{(q)_\infty}{(a/q)_\infty} \sum_{n=1}^{\infty} \sum_{r=1}^k {n \choose r} {k-1 \choose k-r}  c^{n} q^n \frac{(a/q)_{n}}{(q)_{n}} \\
&= \frac{1}{cq} \frac{(q)_\infty}{(a/q)_\infty} \sum_{r=1}^k {k-1 \choose k-r}  \sum_{n=1}^{\infty} {n \choose r} \frac{(a/q)_{n}}{(q)_{n}} c^n q^n. 
\end{align*}
Now we apply Lemma \ref{lemma 5} in the right side of the above expression to see that
\begin{align*}
\sum_{n=0}^{\infty}\frac{(q^2/a)_n a^n}{(1-cq^{n+1})^{k+1}(q)_n} = \frac{1}{c q} \frac{(q)_\infty}{(a/q)_\infty} \sum_{r=1}^k {k-1 \choose k-r} \frac{1}{r!}\left[\frac{d^{r}}{dx^{r}}\frac{(x ac)_\infty}{(x c q)_\infty}\right]_{x=1}. 
\end{align*}
At this moment,  employing Lemma \ref{lemma 6},  we obtain
\begin{align*}
\sum_{n=0}^{\infty}\frac{(q^2/a)_n a^n}{(1-cq^{n+1})^{k+1}(q)_n} &= \frac{1}{cq} \frac{(q)_\infty}{(a/q)_\infty} \sum_{r=1}^k {k-1 \choose k-r} \frac{1}{r!} \left[\frac{(ac )_\infty}{(c q)_\infty} N_r(T_{1,a,c},T_{2,a,c},\cdots, T_{r,a,c})\right]. 
\end{align*}
Now utilizing Lemma \ref{lemma 7},  we can find a polynomial $P_k[x_1,  x_2,  \dots,  x_k] \in \mathbb{Q}[x_1,  x_2,  \dots,  x_k]$ of degree $k$ such that
\begin{align*}
\sum_{n=0}^{\infty}\frac{(q^2/a)_n a^n}{(1-cq^{n+1})^{k+1}(q)_n} = \frac{1}{c q} \frac{(q)_\infty (ac)_\infty}{(a/q)_\infty (c q)_\infty}P_{k}(\mathfrak{S}_{0,a,c}(q),\mathfrak{S}_{1,a,c}(q), \dots, \mathfrak{S}_{k-1,a,c}(q)).
\end{align*}
Finally,  multiplying both sides by $(1-a/q)$ and upon simplification,  we  complete the proof of Theorem \ref{polynomial version generalization of gupta-kumar}.
\end{proof}

\begin{proof}[Theorem \ref{new expression to ABEM}][]
Letting $a \rightarrow 0$ and $\alpha =1$ in Corollary \ref{2-var of RR} gives 
\begin{align*}
\sum_{n=0}^\infty \frac{(-1)^n q^{\frac{n(n-1)}{2}} b^n}{(q)_n(1-c q^n)} = \sum_{n=0}^\infty c^n (b q^n)_\infty.  
\end{align*}
Further,  replacing $c$ by $cq$ and $b$ by $q^2$,  the above identity reduces to 
\begin{align*}
\sum_{n=0}^\infty \frac{(-1)^n q^{\frac{n(n+1)}{2}} q^n}{(q)_n(1-c q^{n+1})} = \sum_{n=0}^\infty c^n q^n (q^{n+2})_\infty.  
\end{align*}
Changing the variable $n$ by $n-1$ on both sides and simplifying,  we get
\begin{align*}
\sum_{n=1}^\infty \frac{(-1)^{n-1} q^{\frac{n(n-1)}{2}} }{(q)_{n-1}}  \frac{c q^n}{(1-c q^{n})}= \sum_{n=1}^\infty c^n q^n (q^{n+1})_\infty.  
\end{align*}
Now apply the differential operator $D=c\frac{d}{dc}$,  $m$-times on both sides of the above identity to obtain
\begin{align}\label{identity with differential operator}
\sum_{n=1}^\infty \frac{(-1)^{n-1}q^{n \choose 2}}{(q)_{n-1}}D^m\left(\frac{cq^n}{1-cq^n}\right)= \sum_{n=1}^\infty n^m c^n q^n(q^{n+1})_\infty. 
\end{align}
To simplify further,  we shall make use of the following property of Eulerian polynomials: 
\begin{align*}
\frac{A_m(x)}{(1-x)^{m+1}}=\sum_{j=0}^\infty x^j(j+1)^m.
\end{align*}
Letting $x=c q^n$ in the above identity,  one can easily verify that
\begin{align}\label{Eulerian polynomial relation}
\frac{cq^nA_m(cq^n)}{(1-cq^n)^{m+1}} = D^m\left(\frac{cq^n}{1-cq^n}\right).
\end{align}
Now substituting \eqref{Eulerian polynomial relation} in \eqref{identity with differential operator},   we obtain
\begin{align*}
 \sum_{n=1}^\infty \frac{(-1)^{n-1}q^{n\choose 2}cq^n A_m(cq^n)}{(1-cq^n)^{m+1}(q)_{n-1}} = \sum_{n=1}^\infty n^m c^n q^n(q^{n+1})_\infty. 
\end{align*}
This completes the proof of the first equality in Corollary \ref{2-var of RR}.  The second equality has already been proved by the present authors in \cite[Theorem 2.2]{ABEM} (see \eqref{2nd gen_Uchimura by ABEM}).    
\end{proof}
\begin{proof}[Corollary \ref{Entry 4 with new term}][]
Letting $m=1$ in Theorem \ref{new expression to ABEM},  it simplifies to 
\begin{align*}
\sum_{n=1}^\infty n c^n q^n(q^{n+1})_\infty = \sum_{n=1}^\infty \frac{(-1)^{n-1} q^{{n \choose 2}}(cq^n)A_1(cq^n)}{(1-cq^n)^2 (q)_{n-1}}=\frac{(q)_\infty}{(cq)_\infty}Y_1(K_{1,c}(q)).
\end{align*}
Note that the Eulerian polynomial $A_1(x)=1$ and the Bell polynomial $Y_1(x) = x$.  Moreover,  using the definition \eqref{M_m and K_(m+1)} of $K_{1,c}(q)$,  the above identity reduces to 
\begin{align*}
\sum_{n=1}^\infty n c^n q^n(q^{n+1})_\infty = \sum_{n=1}^\infty \frac{(-1)^{n-1} q^{{n \choose 2}+n}c}{(1-cq^n)^2 (q)_{n-1}}= \frac{(q)_\infty}{(cq)_\infty} \sum_{n=1}^{\infty} \sigma_{0, c}(n) q^n. 
\end{align*} 
Further,  multiplying by $ \frac{(cq)_\infty}{(q)_\infty}$ throughout the identity and using the definition of $\sigma_{0,c}(n)$,  we obtain
\begin{align*}
  \frac{(cq)_\infty}{(q)_\infty} \sum_{n=1}^\infty n c^n q^n(q^{n+1})_\infty = \frac{(cq)_\infty}{(q)_\infty} \sum_{n=1}^\infty \frac{(-1)^{n-1}cq^{{\frac{n(n+1)}{2}}}}{(1-cq^n)^2 (q)_{n-1}}=\sum_{n=1}^\infty \frac{c^n q^n}{1-q^n}.
\end{align*}
Now our main aim is to simplify the middle term of the above identity since the left most and  the right most expressions are exactly in the same form as we wanted in \eqref{Uchimura-type for Ramanujan's identity}.  Changing the variable $n$ by $n+1$,  the middle expression can be rewritten in the following way:
\begin{align*}
& \frac{(cq)_\infty}{(q)_\infty} \sum_{n=0}^\infty \frac{(-1)^{n}cq^{{\frac{(n+1)(n+2)}{2}}}}{(1-cq^{n+1})^2 (q)_{n}} \\ 
& = cq \frac{(cq)_\infty}{(q)_\infty} \sum_{n=0}^\infty \frac{(-1)^{n}q^{\frac{n(n+1)}{2}} (cq^2)_{n-1}^2 q^n}{(q)_{n} (cq^2)_n^2} \\
&=\frac{(cq)_\infty}{(q)_\infty}\frac{cq}{(1-cq)^2} \lim_{\gamma \rightarrow 0} \sum_{n=0}^\infty \frac{(cq)_{n} (cq)_n (q/\gamma)_n (\gamma q)^n}{(cq^2)_n(cq^2)_n(q)_{n}}, \\ &=\frac{(cq)_\infty}{(q)_\infty}\frac{cq}{(1-cq)^2}  \lim_{\gamma \rightarrow 0} {}_{3}\phi_{2}\left[\begin{matrix} q/\gamma, & cq, & cq \\ & cq^2, & cq^2 \end{matrix} \, ; q, \gamma q  \right].
\end{align*}
Now applying ${}_{3}\phi_{2}$ transformation formula  \eqref{3_phi_2} to the above expression yields that
\begin{align*}
\frac{(cq)_\infty}{(q)_\infty}\frac{cq}{(1-cq)^2}  \lim_{\gamma \rightarrow 0} \frac{(\gamma cq)_\infty (q^2)_\infty}{(cq^2)_\infty(\gamma q)_\infty }  {}_{3}\phi_{2}\left[\begin{matrix} q/\gamma, & q, & q \\ & cq^2, & q^2 \end{matrix} \, ; q, \gamma cq  \right].
\end{align*}
Further,  upon simplification,  it reduces to 
\begin{align*}
\frac{1}{(1-q)}\frac{cq}{(1-cq)} \sum_{n=0}^\infty \frac{(-1)^n (q)_{n} q^{\frac{n(n+1)}{2}}(cq)^{n}  }{(cq^2)_{n} (q^2)_{n}}  
 & =  \sum_{n=0}^\infty \frac{(-1)^n  q^{\frac{n(n+1)}{2}}(cq)^{n+1}  }{(cq)_{n+1} (1 -q^{n+1}) },  \\
& =  \sum_{n=1}^\infty \frac{(-1)^{n-1}  q^{\frac{n(n-1)}{2}}(cq)^{n}  }{(cq)_{n} (1 -q^{n}) },
\end{align*}
which is exactly same as the middle expression in Corollary \ref{Entry 4 with new term}.  This completes the proof.  
\end{proof}

\begin{proof}[Corollary \ref{Uchimura's identity with new expression}][]
The proof of this corollary immediately follows  by substituting $c=1$ in Theorem \ref{new expression to ABEM}.  

\end{proof}

\begin{proof}[Theorem \ref{limit in polynomial form}][]
First,  we define the following  two generating functions for the sequence of polynomials $\{ t_n(a, q) \}$ and $\{f(n) \}$ as 
\begin{align*}
T(a,\alpha,q)&:=\sum_{n=1}^\infty t_n(a,q) \alpha^n,  \\
F(\alpha)&:=\sum_{n=1}^\infty f(n)\alpha^n.
\end{align*}
From \eqref{define t_n}, we see that
\begin{align*}
(1-aq^n)t_n(a,q)=(f(n)-af(n+1)) + (1-q^{n-1})t_{n-1}(a,q).
\end{align*}
Multiplying both sides by $\alpha^n$ and summing over all the natural numbers $n$, we get
\begin{align*}
T(a,\alpha,q)=\frac{(1-a/\alpha)}{(1-\alpha)}F(\alpha) - \frac{(1-a/\alpha)\alpha}{(1-\alpha)}T(a,\alpha q,q).
\end{align*}
The above recurrence relation leads us to
\begin{align*}
T(a,\alpha,q)=\sum_{n=1}^\infty \frac{(-1)^{n-1} (a/\alpha q^{n-1})_n F(\alpha q^{n-1}) \alpha^{n-1} q^{\frac{(n-1)(n-2)}{2}} }{ (\alpha)_n}.
\end{align*}
Substitute $\alpha = q$ to see that
\begin{align}\label{T(a,q,q)}
T(a,q,q)=\sum_{n=1}^\infty t_n(a,q) q^n &= \sum_{n=1}^\infty \frac{(-1)^{n-1} (a/q^{n})_n F(q^{n}) q^{\frac{n(n-1)}{2}}}{ (q)_n} = -\sum_{n=1}^\infty \frac{(q/a)_n F(q^{n}) (a/q)^n}{ (q)_n}.
\end{align}
Utilizing \eqref{define t_n},  it can be shown that
\begin{align*}
t_n(a,q)-t_{n-1}(a,q)= f(n)-af(n+1)-q^{n-1}t_{n-1}(a,q)+aq^nt_n(a,q).
\end{align*}
Taking summation over $n$ from $1$ to $\ell$ on both sides and upon simplification,  one can show that 
\begin{align}\label{limit in term of T(a,q,q)}
\lim_{\ell \rightarrow \infty} \left( \sum_{n=1}^\ell f(n) - \frac{a}{1-a}f(\ell+1) - \left(\frac{1-aq^\ell}{1-a}\right) t_\ell(a,q) \right) = \sum_{n=1}^\i t_n(a,q)q^n=T(a,q,q).
\end{align}
Now we make use of another form of generating function for $f(n)$ given in \cite[p.~50]{andrewssiam1997}, namely,
\begin{align*}
F(\a)=\sum_{m\geq 0}\sum_{k\geq m}c_k \tilde{s}(k,m)m!\frac{\a^m}{(1-\a)^{m+1}}-c_0,
\end{align*}
where $\tilde{s}(k,m)$ denotes the Stirling numbers of the second kind.  
Substituting $\a=q^n$ and upon simplification,  one can show that
\begin{align}\label{F(q^n)}
F(q^n)=\sum_{m\geq 1} d_m\sum_{j=0}^{m-1}e_{m,j} \frac{q^{n}}{(1-q^n)^{m+1-j}}+c_0\frac{q^n}{1-q^n},
\end{align}
where
\begin{align*}
d_m&=\sum_{k\geq m}c_k \tilde{s}(k,m)m!,~\text{for}~m\geq 1,~d_0=c_0,\\
e_{m,j}&=(-1)^j {m-1 \choose j}.
\end{align*}
Finally,  substituting the expression \eqref{F(q^n)} for $F(q^n)$ in \eqref{T(a,q,q)} and then substituting $T(a,  q, q)$ in \eqref{limit in term of T(a,q,q)}, it follows that 
\begin{align}
\lim_{\ell \rightarrow \infty} \left( \sum_{n=1}^\ell f(n) - \frac{a}{1-a}f(\ell+1) - \left(\frac{1-aq^\ell}{1-a}\right) t_\ell(a,q) \right)&= -\sum_{n=1}^\infty \frac{(q/a)_n F(q^{n}) (a/q)^n}{ (q)_n}\label{limit formula in term of F(q^n)}\\
&=-\sum_{m\geq 1} d_m\sum_{j=0}^{m-1}e_{m,j} \sum_{n=1}^\infty \frac{(q/a)_n a^n}{(1-q^n)^{m+1-j}(q)_n}\nonumber\\
&\hspace{1.5cm} -c_0\sum_{n=1}^\i \frac{(q/a)_na^n}{(1-q^n)(q)_n}\nonumber.
\end{align}
Now from the definition \eqref{definition h_j}  of $h_j$ and the polynomial $P_k(a, q)$ defined in \eqref{guptakumar polynomial},  it is clear that
\begin{align*}
\lim_{\ell \rightarrow \infty} \left( \sum_{n=1}^\ell f(n) - \frac{a}{1-a}f(\ell+1) - \left(\frac{1-aq^\ell}{1-a}\right) t_\ell(a,q) \right)&= \sum_{m\geq 1} d_m\sum_{j=0}^{m-1}e_{m,j}P_{m+1-j}(a,q)+c_0 P_1(a,  q)\\
&=\sum_{j=1}^\i h_j P_j(a,  q).
\end{align*}
This completes the proof of Theorem \ref{limit in polynomial form}.
\end{proof}

\begin{proof}[Theorem \ref{limit formula for periodic sequence}][]
The main idea of the proof of this result goes along the same lines as in \cite{Bringmann,  GK21}.  However,  for completeness we give necessary details of the proof.  
As $f(n)$ is a periodic function of period $N$,  one can write 
\begin{align}\label{diff rep f(n)}
f(n)= \frac{1}{N} \sum_{j=1}^N f(j) \sum_{k=0}^{N-1} \z_N^{(n-j)k},
\end{align}
where $\zeta_N= e^{\frac{2\pi i}{N}}$. 
Now using \eqref{diff rep f(n)},  we can show that the generating function for $f(n)$,  for $|\a|<1$,  is 
\begin{align}\label{F(alpha) when alpha less than 1}
F(\a)= \sum_{n=1}^\infty f(n) \a^n =  \sum_{k=0}^{N-1}\frac{c_k\a}{1-\z_N^k \a},
\end{align}
where $c_k = \frac{1}{N} \sum_{j=1}^N f(j) \zeta_N^{(1-j)k}$.  
In the proof of Theorem \ref{limit in polynomial form},  from equation \eqref{limit formula in term of F(q^n)},  we have seen that
\begin{align*}
\lim_{\ell \rightarrow \infty} \left( \sum_{n=1}^\ell f(n) - \frac{a}{1-a}f(\ell+1) - \left(\frac{1-aq^\ell}{1-a}\right) t_\ell(a,q) \right)&= -\sum_{n=1}^\infty \frac{(q/a)_n F(q^{n}) (a/q)^n}{ (q)_n}.
\end{align*}
Now substituting the expression \eqref{F(alpha) when alpha less than 1} for $F(q^n)$,  we see that
\begin{align*}
\lim_{\ell \rightarrow \infty} &\left( \sum_{n=1}^\ell f(n) - \frac{a}{1-a}f(\ell+1) - \left(\frac{1-aq^\ell}{1-a}\right) t_\ell(a,q) \right)= -\sum_{n=1}^\infty \sum_{k=0}^{N-1} c_k \frac{(q/a)_n a^n}{(1-\z_N^k q^n)(q)_n} \\
&\hspace{4cm}=-c_0 \sum_{n=1}^\i \frac{(q/a)_n a^n}{(1-q^n)(q)_n} -\sum_{k=1}^{N-1} c_k \sum_{n=1}^\infty \frac{(q/a)_n a^n}{(1-\z_N^k q^n)(q)_n}\\
&\hspace{4cm}=c_0 P_1( \mathfrak{S}_{0,a}(q) )-\sum_{k=1}^{N-1} \frac{c_k}{(1-\z_N^k)}\left( {}_{2}\phi_{1}\left( \begin{matrix} q/a, & \z_N^k  \\ & \z_N^kq \end{matrix}  ; q; a \right) - 1 \right),
\end{align*}
where $P_1$ is  defined in \eqref{guptakumar polynomial} and one can verify that $P_1(x)=x$.  Finally, using $q$-Gauss summation formula \eqref{q-Gauss sum}, we have
\begin{align*}
\lim_{\ell \rightarrow \infty} \left( \sum_{n=1}^\ell f(n) - \frac{a}{1-a}f(\ell+1) - \left(\frac{1-aq^\ell}{1-a}\right) t_\ell(a,q) \right) & = c_0 \mathfrak{S}_{0,a}(q) + \sum_{k=1}^{N-1} \frac{c_k}{(1-\z_N^k)} \\
&  - \frac{(q)_\i}{(a)_\i} \sum_{k=1}^{N-1}c_k \frac{(a\z_N^k)_\i}{(\z_N^k)_\i}.
\end{align*}
This finishes the proof of Theorem \ref{limit formula for periodic sequence}.
\end{proof}

\begin{proof}[Corollary \ref{simpler version of limit value}][]

The right hand side expression of this corollary has already been derived in  \cite[Corollary 4.6]{GK21},  whereas the left side expression is obtained from Theorem \ref{limit formula for periodic sequence}.  This finishes the proof of this corollary.

\end{proof}

\section{Concluding Remarks}
One of the main goals of this paper was to study different generalizations of Uchimura's identity \eqref{Uchimura}.  We named three sides of Uchimura's identity \eqref{Uchimura}  as {\it Uchimura-type} sum,  {\it Ramanujan-type} sum and {\it divisor-type} sum.  In view of these three sides,  we tried to establish a unified theory for generalizations given by Uchimura \eqref{uchimura ideantity polynomial version},  by the present authors \eqref{2nd gen_Uchimura by ABEM},  Dilcher \eqref{dilcher 1},  Andrews-Crippa-Simon \eqref{dilcher 2},  \eqref{ACS}, and Gupta-Kumar \eqref{guptakumar}, \eqref{guptakumar polynomial}.  We showed that all these generalizations can be obtained from our general identity,  i.e.,  Theorem \ref{General F function}.  For example,  at a glance,  it may look as if Uchimura's identity \eqref{Uchimura} and Dilcher's identity \eqref{dilcher 1} are not branches of a bigger tree,  however,  we have shown that both of these identities can be obtained from a single identity,  namely,   Corollary \ref{Uchimura generalization_2 variable}. 
We obtained a one-variable generalization of Gupta-Kumar's identity \eqref{guptakumar polynomial} (see Theorem \eqref{polynomial version generalization of gupta-kumar}),  which is also a two variable generalization of the identity \eqref{dilcher 2} of Andrews,  Crippa and Simon.  

Further,  we established a {\it Ramanun-type} expression for our generalized identity \eqref{2nd gen_Uchimura by ABEM},  and as a by-product,  we found a {\it Ramanujan-type} sum for Uchimura's generalization \eqref{uchimura ideantity polynomial version} and also derived a {\it Uchimura-type} sum for Ramanujan's identity \eqref{entry4}.  

We also corrected an identity \eqref{dilcher's wrong identity} of Dilcher and while correcting it,  we found two sequences of rational numbers involving Stirling numbers of the first kind, namely,  $A(j,r,t)$ \eqref{A(j,r,i)} and $C(k,r,t)$ \eqref{define C(k,r,t)}.  Quite interestingly,  we observed that these numbers obey interesting recurrence relations. (see Theorem \ref{recurrence relation of C(k,r,t) and A(j,r,t)}). 

In Section \ref{application},  we discussed applications of the generalizations of Uchimura's  identities \eqref{Uchimura} in probability theory given by Uchimura and connection with theory of random acyclic digraphs developed by Simon-Crippa-Collenberg and Andrews-Crippa-Simon.    We gave one variable generalization of an identity of Andrews-Crippa-Simon (see Theorem \ref{limit in polynomial form}),  and a one variable generalization of an identity  of Bringmann and Jennings-Shaffer  (see Theorem \ref{limit formula for periodic sequence}), 
which is a corrected version of an identity of Gupta-Kumar.  


Now we state a few problems that might be of interest to the readers of this article. 

\begin{enumerate}
\item Note that Corollary \ref{two variable gen_Uchimura} is a simultaneous generalization of Uchimura's  identity \eqref{Uchimura} and Dilcher's identity \eqref{dilcher 1}.  However it gives only the first equality.  Thus,  a natural question is that `what will be the {\it divisor-type} sum for Corollary \ref{two variable gen_Uchimura}?'

\item Utilizing \eqref{uchimura ideantity polynomial version},  Uchimura showed that,  if $\textup{Pr}(X=n)= q^n (q^{n+1})_\infty$,  then
\begin{align*}
E(X)= \sum_{n=1}^\infty d(n) q^n,  \textup{Var}(X)= \sum_{n=1}^\infty \sigma(n) q^n.
\end{align*} 
More generally,  he showed that,  for any $m \in \mathbb{N}$, 
\begin{align}\label{mth cum}
h_m= m\textrm{-th cumulant} = \sum_{n=1}^\infty \sigma_{m-1}(n) q^n.  
\end{align}
Utilizing \eqref{Uchimura},  Simon,  Crippa and Collenberg established that,  if $\text{Pr}(\gamma_{n}^{*}(1) = h )$ is defined as in \eqref{Prob},  then 
\begin{align*}
\lim_{n \rightarrow \infty} \left( n - E(\gamma_{n}^{*}(1)\right) = \sum_{n=1}^\infty d(n) q^n,
\end{align*}
Later,  employing \eqref{uchimura ideantity polynomial version}, Andrews,  Crippa and Simon proved that 
\begin{align*}
\lim_{n \rightarrow \infty} \text{Var}(\gamma_{n}^{*}(1)) = \sum_{n=1}^\infty \sigma(n) q^n. 
\end{align*}
Thus,  it would be an interesting problem to know the interpretation of the generalized divisor generating function \eqref{mth cum} with respect to the random variable $\gamma_{n}^{*}(1)$. 


\item As an application of Uchimura's identity \eqref{uchimura ideantity polynomial version},  Andrews,  Crippa and Simon proved the identity \eqref{andrews crippa simmos's result with limit}.  In a similar way,  as an application of Gupta-Kumar's identity \eqref{guptakumar polynomial},  we obtained Theorem \ref{limit in polynomial form}.  In this paper,  we obtained a one variable generalization of \eqref{guptakumar polynomial},  i.e.,  Theorem \ref{polynomial version generalization of gupta-kumar}.  Thus,  it would be an exciting problem to find an application of Theorem \ref{polynomial version generalization of gupta-kumar}, which will generalize \eqref{andrews crippa simmos's result with limit} as well as Theorem \ref{limit in polynomial form}.  


\item  
By combining a result of van Hamme \cite{hamme} and Guo and Zeng  \cite[Equation (3.8)]{guozeng2015},  we get a finite analogue  of Uchimura's identity \eqref{Uchimura},  namely, for any $N \in \mathbb{N}$,  
\begin{align}\label{finite analogue}
\sum_{n=1}^\infty n q^n (q^{n+1})_{N-1} - \sum_{n=1}^\infty n q^{n+N} (q^{n+1})_{N-1}= \sum_{n=1}^{N} \frac{(-1)^{n-1} q^{{n+1 \choose 2}}}{1- q^n} \left[\begin{matrix} N \\ n \end{matrix}\right] = \sum_{n=1}^{N} \frac{q^n}{1-q^n}. 
\end{align}
In this paper,  we discussed various generalizations of Uchimura's identity \eqref{Uchimura}.  Thus,  it would be interesting to see finite analogues of these generalizations.  Readers are encouraged to see \cite{DEMS} to know more about finite analogues of some of the results motivated from the above identity \eqref{finite analogue}.  
\end{enumerate}

\section{Acknowledgement}
The first author wants to thank University Grant Commission (UGC), India,  for providing Ph.D.  scholarship.  The second author sincerely thanks the Department of Mathematics, Pt. Chiranji
Lal Sharma Government College, Karnal,  for a conducive research environment. 
The third author wants to thank BITS Pilani for providing the New Faculty Seed Grant NFSG/PIL/2024/P3797.  
The last author is thankful to the Science and Engineering Research Board (SERB),  India,  for giving the Core Research Grant CRG/2023/002122 and MATRICS Grant MTR/2022/000545.  The authors would also like to thank Dr.  Debopriya Mukherjee and Dr. Priyamvada for fruitful discussions in probability theory and graph theory during this project.


\begin{thebibliography}{00}

\bibitem{ABEM} A.~Agarwal, S.~C.~Bhoria, P.~Eyyunni and B.~Maji, \emph{Bressoud-Subbarao type weighted partition identities for a generalized divisor function},  Ann. Comb. (2023).  https://doi.org/10.1007/s00026-023-00647-1
 
 
 
 
 
 
 \bibitem{andrews98}
  G.~E.~Andrews, The Theory of Partitions, Addison-Wesley Pub. Co., NY, 300 pp. (1976). Reissued, Cambridge University Press, New York, 1998.
 
 
 
 
\bibitem{andrews08} G.~E.~Andrews, \emph{The number of smallest parts in the partitions on $n$}, J. Reine Angew. Math.  {\bf 624} (2008), 133--142.
 
 
 
  
  
 
%
 
\bibitem{andrewssiam1997}
G.~E.~Andrews, D.~Crippa and K.~Simon, \emph{$q$-series arising from the study of random graphs}, SIAM J.~Discrete~Math.~\textbf{10} (1997), 41--56.
 
 
\bibitem{agl13} G.~E.~Andrews, F.~G.~Garvan,  and J. Liang, \emph{Self-conjugate vector partitions and the parity of the \textup{spt}-function}, Acta Arith.  \textbf{158} (2013), 199--218.  
 
 
 
 %
 
\bibitem{bcbramforthnote}
B.~C.~Berndt,  Ramanujan's Notebooks, Part IV,  Springer-Verlag,    New York, 1991.
 
 
 
 \bibitem{BEM21}
 S. ~C.~Bhoria,  P.~Eyyunni and B.~Maji,  \emph{Generalization of five $q$-series identities of Ramanujan and unexplored weighted partition identities},  Ramanujan J.  {\bf 58} (2022),  435--462.  
 
 \bibitem{Bringmann}
 K.~Bringmann and C.~Jennings-Shaffer, \emph{Some q-series identities extending work of Andrews, Crippa, and Simon on sums of divisors functions}, ~Discrete Math. ~\textbf{343} (2020), 112019.
  
 
 \bibitem{bresub}
 D.~Bressoud and M.~Subbarao, \emph{On Uchimura's connection between partitions and the number of divisors}, Canad.~Math.~Bull.  ~\textbf{27} (1984), 143--145.
 
 
 \bibitem{comtet}
 L.~Comtet,  Advanced Combinatorics, Reidel, Dordrecht, 1974.  
 
 
\bibitem{dilcher} 
K.~Dilcher, \emph{Some $q$-series identities related to divisor functions}, Discrete Math.  {\bf 145} (1995), 83--93.  
 
 

\bibitem{DM}
 A.~Dixit and B.~Maji, \emph{Partition implications of a three-parameter $q$-series identity}, Ramanujan J.  {\bf 52} (2020), 323--358.

 \bibitem{DEMS} A.~Dixit, P.~Eyyunni, B.~Maji and G.~Sood, \emph{Untrodden pathways in the theory of the restricted partition function p(n,N)},  J.  Combin.  Theory Ser.  A {\bf 180} (2021), 105423.
 

 
 \bibitem{fine} N.~J.~Fine, \emph{Basic Hypergeometric Series and Applications}, Mathematical Surveys and Monographs, Amer. Math. Soc., 1989. 
 
 \bibitem{ffw95} R.~Fokkink, W.~Fokkink and Z.~B.~Wang, \emph{A relation between partitions and the number of divisors}, Amer. Math. Monthly {\bf 102} (1995), 345--347.
 
 
 
 \bibitem{garvan1} F.~G.~Garvan, \emph{Weighted partition identities and divisor sums}, Ch. 12 in Frontiers in Orthogonal Polynomials and $q$-Series, Contemporary Mathematics and Its Applications: Monographs, Expositions and Lecture Notes: Vol. 1, M.~Z.~Nashed and X.~Li, eds., World Scientific, 2018, pp. 239--249.
 
\bibitem{GasperRahman} G.~Gasper and M.~Rahman, Basic Hypergeometric series, Second Edition, (Encyclopedia of Mathematics and its applications), 2004. 
 

\bibitem{guozeng2015} V.~J.~W.~Guo and J.~Zeng, \emph{Basic and bibasic identities related to divisor functions}, 
J. Math. Anal. Appl.  {\bf 431} (2015), 1197--1209.
 
 \bibitem{GK21} R.~Gupta and R.~Kumar,  \emph{On some $q$-series identities related to a generalized divisor function and their implications},  Discrete Math.  {\bf 344} (2021),  112559.  
 
 
 
 \bibitem{hamme} L.~van Hamme, \emph{Problem $6407$}, Amer. Math. Monthly {\bf 89} No. 9 (1982), 703--704.
 
 
 
 
 \bibitem{kluyver} J.~C.~Kluyver, Vraagstuk XXXVII (Solution by S.C. van Veen), Wiskundige Opgaven (1919), 92--93.
 
\bibitem{Palmer85} E.~Palmer,  Graphical Evolution,  Wiley Interscience Publishers,  New York, 1985.   
 
 
 
 
 \bibitem{ramanujanoriginalnotebook2}
 S.~Ramanujan,  Notebooks of Srinivasa Ramanujan, Vol. II, Tata Institute of Fundamental Research, Bombay, 1957, 393 pp.


 
 \bibitem{ramanujantifr}
  S.~Ramanujan,  Notebooks of Srinivasa Ramanujan, Vol. II, Tata Institute of Fundamental Research, Mumbai, 2012.
  

\bibitem{SCC1993}
K.~Simon, D.~Crippa and F.~Collenberg \emph{On the distribution of the transitive closure in a random acyclic digraph}in: Algorithms - ESA ’93, in: Lecture Notes in Computer Science, vol. 726, Springer-Verlag, Berlin, 1993, pp. 345--356.




 
\bibitem{uchimura81} K.~Uchimura, \emph{An identity for the divisor generating function arising from sorting theory}, J. Combin.
 Theory Ser. A {\bf 31} (1981), 131--135. 
 
\bibitem{uchimura87} K.~Uchimura,  \emph{Divisor generating functions and insertion into a heap},  Discrete Appl. Math.  {\bf 18} (1987), 73--81.   
 
 
 
 
 \end{thebibliography}
\end{document}